\documentclass{amsart}
\usepackage{amsthm}
\usepackage{mathtools}
\usepackage{xcolor}
\usepackage{amssymb}
\usepackage{amsmath}
\usepackage{tabularx}
\usepackage{graphicx}
\usepackage{float}
\usepackage{placeins}
 \usepackage[all]{xy}
\usepackage{thmtools,thm-restate}
\usepackage{todonotes}
\usepackage{tikz}
\usepackage{mathpazo}
\usepackage{lscape}
\usepackage{enumerate}
\usepackage{subcaption}
\usepackage[pagebackref,colorlinks=true,
	urlcolor=blue!20!black,
	citecolor=green!50!black]{hyperref}
\usepackage[nobysame]{amsrefs}
\makeatletter
\let\orgdescriptionlabel\descriptionlabel
\renewcommand*{\descriptionlabel}[1]{%
  \let\orglabel\label
  \let\label\@gobble
  \phantomsection
  \protected@edef\@currentlabel{#1\unskip}%
  \let\label\orglabel
  \orgdescriptionlabel{(#1)}%
}
\makeatother
\makeatletter
\@namedef{subjclassname@2020}{%
  \textup{2020} Mathematics Subject Classification}
\makeatother

\theoremstyle{plain}
\newtheorem{theorem}{Theorem}[section]
\newtheorem{lemma}[theorem]{Lemma}
\newtheorem{proposition}[theorem]{Proposition}
\newtheorem{corollary}[theorem]{Corollary}

\newtheorem{conjecture}[theorem]{Conjecture}
\theoremstyle{remark}

\newtheorem{example}[theorem]{Example}
\newtheorem{remark}[theorem]{Remark}

\newcommand{\defn}[1]{{\color{green!50!black}\emph{#1}}}
\newcommand{\ie}{\text{i.e.}\;}
\newcommand{\defs}{\overset{\mathsf{def}}{=}}
\newcommand{\bubcov}{\lessdot_{\mathsf{bub}}}
\newcommand{\bubless}{<_{\mathsf{bub}}}
\newcommand{\bubleq}{\leq_{\mathsf{bub}}}
\newcommand{\indel}{\rightarrow}

\newcommand{\shufleq}{\leq_{\mathsf{shuf}}}

\newcommand{\sindel}{\hookrightarrow}
\newcommand{\transpose}{\Rightarrow}
\newcommand{\Poset}{\mathbf{P}}
\newcommand{\Lattice}{\mathbf{L}}

\newcommand{\indeg}{\mathsf{in}}
\newcommand{\indeld}{\mathsf{in_{\sindel}}}
\newcommand{\transd}{\mathsf{in_{\transpose}}}
\DeclareMathOperator{\rk}{rk} 
\DeclareMathOperator{\cork}{cork} 

\newcommand{\lnk}{\mathsf{link}}
\newcommand{\dlt}{\mathsf{del}}
\newcommand{\invset}{\mathsf{Inv}}
\newcommand{\Shuf}{\mathsf{Shuf}}

\newcommand{\ShufPoset}{\textbf{\textsf{Shuf}}}
\newcommand{\Bub}{\textbf{\textsf{Bub}}}
\newcommand{\Hoch}{\textbf{\textsf{Hoch}}}
\DeclareMathOperator{\Res}{\mathsf{Res}} 
\newcommand{\Del}{\mathsf{Del}}
\newcommand{\ch}{\tilde{\mathsf{c}}\mathsf{h}}
\def\ubf{\mathbf{u}}\def\vbf{\mathbf{v}}\def\wbf{\mathbf{w}}\def\xbf{\mathbf{x}}\def\ybf{\mathbf{y}}

\def\Ecal{\mathcal{E}}\def\Lcal{\mathcal{L}}\def\Tcal{\mathcal{T}}\def\Ucal{\mathcal{U}}

\def\pr{\prime}\def\setm{\setminus}\def\wtil{\widetilde}

\def\rd{\textcolor{red}}
\def\bl{\textcolor{blue}}

\title{Bubble Lattices II: Combinatorics}
\author[T.~McConville]{Thomas McConville}
\address[T.~McConville]{Kennesaw State University, Department of Mathematics, 30144 Kennesaw (GA), USA}
\email{tmcconvi@kennesaw.edu}
\author[H.~M\"uhle]{Henri M{\"u}hle}
\address[H.~M\"uhle]{Qoniac GmbH, Dr.-K{\"u}lz-Ring 15, 01067 Dresden, Germany.}
\email{henri.muehle@proton.me}
\keywords{shuffle lattice, bubble lattice, $F$-triangle, $H$-triangle, $M$-triangle, noncrossing matching complex, noncrossing bipartite complex}
\subjclass[2020]{05E45, 06A07, 06B99}

\begin{document}

\allowdisplaybreaks

\begin{abstract}
	We introduce two simplicial complexes, the noncrossing matching complex and the noncrossing bipartite complex.  Both complexes are intimately related to the bubble lattice introduced in our earlier article ``Bubble Lattices I: Structure'' (arXiv:2202.02874).  We study these complexes from both an enumerative and a geometric point of view.  In particular, we prove that these complexes are shellable and give explicit formulas for certain refined face numbers.  Lastly, we conjecture an intriguing connection of these refined face numbers to the so-called $M$-triangle of the shuffle lattice.
\end{abstract}

\maketitle

\section{Introduction}

Fix words $\xbf = x_1x_2\cdots x_m$ and $\ybf = y_1y_2\cdots y_n$ where all $m+n$ letters are distinct. We consider all ways to transform $\xbf$ into $\ybf$ using two types of mutations. In the first type of mutation, known as an \defn{indel} $\ubf\indel \vbf$, the word $\vbf$ is obtained by either deleting some $x_i$ or inserting some $y_j$ into $\ubf$. The \defn{shuffle lattice} is the poset of intermediate words, called \defn{shuffle words}, that appear between $\xbf$ and $\ybf$ by some sequence of indel mutations.

In the second type of mutation, known as a \defn{(forward) transposition} $\ubf\transpose \vbf$, the word $\vbf$ is obtained from $\ubf$ by reversing an adjacent pair $x_iy_j$. We defined the \defn{bubble lattice} in \cite{mcconville.muehle:bubbleI} as the poset of shuffle words that uses both indels and transpositions. More precise definitions are given in Section~\ref{subsec:shuffle_words}.

The shuffle lattice was introduced by Greene in \cite{greene:shuffle}, in which he discovered some remarkable coincidences among its characteristic polynomial, its zeta polynomial, its rank generating function, and other combinatorial invariants. In the current paper, we study two-parameter refinements of the characteristic polynomial and the rank generating function of the shuffle lattice, which we call the \defn{$M$-triangle} and the \defn{$H$-triangle}, respectively.  We also introduce a third polynomial, the \defn{$F$-triangle}, that arises from a refined face enumeration of a certain simplicial complex, the \defn{noncrossing bipartite complex}, whose facets biject to shuffle words and in which two such facets intersect in a codimension-$1$ face if and only if the corresponding shuffle words form a covering pair in the bubble lattice.  We exhibit an intriguing relationship between these three polynomials, which asserts that each of these polynomials can be obtained from the others by a simple variable substitution.

\begin{restatable}{theorem}{fhidentity}\label{thm:FH_identity}
	For all $m,n\geq 0$, 
	\begin{displaymath}
		F_{m,n}(q,t) = q^{m+n} H_{m,n}\left(\frac{q+1}{q}, \frac{t+1}{q+1}\right).
	\end{displaymath}
\end{restatable}

\begin{restatable}{conjecture}{hmidentity}\label{conj:HM_identity}
    For all $m,n\geq 0$,
    \begin{displaymath}
        M_{m,n}(q,t) = (1-t)^{m+n} H_{m,n}\left(\frac{(q-1)t}{1-t},\frac{q}{q-1}\right).
    \end{displaymath}
\end{restatable}

Readers familiar with Coxeter--Catalan combinatorics may recognize the relations from Theorem~\ref{thm:FH_identity} and Conjecture~\ref{conj:HM_identity} from similarly defined polynomials associated with cluster complexes, noncrossing partition lattices and Cambrian lattices.  We wish to emphasize that the bubble and shuffle lattices considered here behave like counterparts to Cambrian lattices and noncrossing partition lattices, respectively.  The newly introduced noncrossing bipartite complex then takes the place of the cluster complex.

To define the $H$-triangle, we introduce in Section~\ref{sec:match} an auxiliary simplicial complex, the \defn{noncrossing matching complex} $\Gamma(m,n)$. Briefly, a face of this complex is a subgraph of a complete bipartite graph with loops whose edges do not cross and whose vertices are incident to at most one edge. The $H$-triangle is essentially the $f$-polynomial of this complex with an additional parameter to keep track of loops. This combinatorial model allows us to explicitly compute the $H$-triangle.

\begin{restatable}{theorem}{bubblehtriangle}\label{thm:bubble_h_triangle}
	For $m,n\geq 0$, we have
	\begin{displaymath}
		H_{m,n}(q,t) = \sum_{a=0}^{\min\{m,n\}} \binom{m}{a}\binom{n}{a}q^a(qt+1)^{m+n-2a}.
	\end{displaymath}
\end{restatable}

One of our main results in \cite{mcconville.muehle:bubbleI} is that the bubble lattice is semidistributive. Consequently, its elements, the shuffle words, are in a natural bijection with the faces of a simplicial complex called the canonical join complex. In Section~\ref{sec:match}, we prove that the noncrossing matching complex is a combinatorial realization of the canonical join complex of the bubble lattice. Canonical join complexes of semidistributive lattices do not automatically inherit nice topological structure. Nevertheless, we show that the noncrossing matching complex is (non-pure) shellable, and it is either contractible or homotopy equivalent to a sphere, see Section~\ref{sec:ncm_vertex_decomposable}.

We also study the subcomplex of $\Gamma(m,n)$ that consists of all faces without loops, and we show that flags of faces of this \defn{positive} part of $\Gamma(m,n)$ are in bijection with certain colored lattice paths, see Section~\ref{sec:ncm_positive}.

The order complex of every interval of the bubble lattice is either contractible or homotopy equivalent to a sphere. The noncontractible intervals correspond 
precisely to the faces of the above-mentioned noncrossing bipartite complex.  Combinatorially,
the faces of this complex again correspond to certain subgraphs of a complete bipartite graph with loops. These subgraphs are still noncrossing, but they are allowed to have multiple edges meeting at a vertex.  

\begin{restatable}{theorem}{deltatopology}\label{thm:delta_topology}
	  For $m,n\geq 0$, the noncrossing bipartite complex $\Delta(m,n)$ is a pure, thin, shellable sphere. 
\end{restatable}

The topological properties of the noncrossing bipartite complex help us to prove Theorem~\ref{thm:FH_identity} geometrically.  The details can be found in Section~\ref{sec:delta_enumeration}, where we also provide a formula for the $F$-triangle: Theorem~\ref{thm:Ftriangle}.

In the last part of this article, we turn our attention to the shuffle lattice and formally introduce the $M$-triangle as a bivariate rank-generating function weighted by the M{\"o}bius function.  We provide some evidence to support Conjecture~\ref{conj:HM_identity} and relate the number of facets of the positive part of the noncrossing bipartite complex to the M{\"o}bius invariant of the shuffle lattice.

\section{Preliminaries}

\subsection{Shuffle Words}\label{subsec:shuffle_words}

For nonnegative integers $m$ and $n$, we consider two disjoint sets of letters: 
\begin{displaymath}
	X = \{x_{1},x_{2},\ldots,x_{m}\}\quad\text{and}\quad Y=\{y_{1},y_{2},\ldots,y_{n}\}.
\end{displaymath}
The central objects of interest in this article are \emph{order-preserving}, \emph{simple} words over $X\cup Y$, \ie words using letters from $X$ or $Y$ without duplicates in which $x_i$ appears before $x_{j}$ (resp. $y_{i}$ appears before $y_{j}$) whenever $i<j$.  We call such words \defn{shuffle words}, and the set of all shuffle words is denoted by $\Shuf(m,n)$.  For example, the word $\bl{y_1y_3}\rd{x_2}\bl{y_4}\rd{x_5x_7x_8}$ belongs to $\Shuf(8,4)$, but the words $\bl{y_1y_1}\rd{x_2}\bl{y_4}\rd{x_5x_7x_8}$ and $\bl{y_1y_3}\rd{x_2}\bl{y_4}\rd{x_7x_5x_8}$ do not.

A shuffle word $\ubf=u_{1}u_{2}\cdots u_{k}\in\Shuf(m,n)$ is uniquely defined by its \defn{interface}, \ie the set of letters $x\in X$ and $y\in Y$ such that there exists some $i\in[k-1]\defs\{1,2,\ldots,k-1\}$ with $u_{i}=y$ and $u_{i+1}=x$, and its \defn{residue}, \ie the set of letters of $\ubf$ which are not in the interface.  

Shuffle words were introduced and studied in \cite{greene:shuffle} as an idealized model for DNA mutation.  Moreover, it was established that the \defn{indel}-operation on $\Shuf(m,n)$, \ie the operation of either deleting a letter from $X$ or inserting a letter from $Y$, equips $\Shuf(m,n)$ with a lattice structure; the \defn{shuffle lattice} $\ShufPoset(m,n)$.  Its order relation is the reflexive and transitive closure of the indel operation, denoted by $\shufleq$.

In \cite{mcconville.muehle:bubbleI}, we have introduced a new partial order on $\Shuf(m,n)$ which results in the \defn{bubble lattice} $\Bub(m,n)$.  It turns out that $\Bub(m,n)$ arises as an order extension from $\ShufPoset(m,n)$, where besides indels, we are also allowed to perform \defn{(forward) transpositions}, \ie we may replace a consecutive subword $xy$ by $yx$.  The order relation of $\Bub(m,n)$ is then the reflexive and transitive closure of the indel and the transposition operation, denoted by $\bubleq$.

In this article, we exhibit further combinatorial properties that connect both lattices.  Figure~\ref{fig:shuffle_bubble} shows the posets $\ShufPoset(2,1)$ and $\Bub(2,1)$.

\begin{figure}
	\centering
	\begin{subfigure}[t]{.45\textwidth}
		\centering
		\includegraphics[page=4,scale=1]{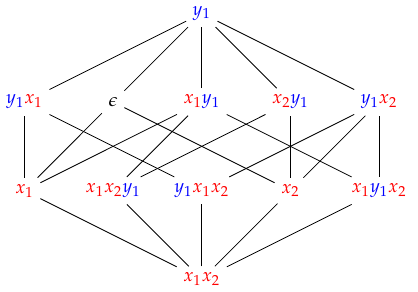}
		\caption{The shuffle poset $\ShufPoset(2,1)$.}
		\label{fig:shuffle_21}
	\end{subfigure}
	\hspace*{1cm}
	\begin{subfigure}[t]{.45\textwidth}
		\centering
		\includegraphics[page=5,scale=1]{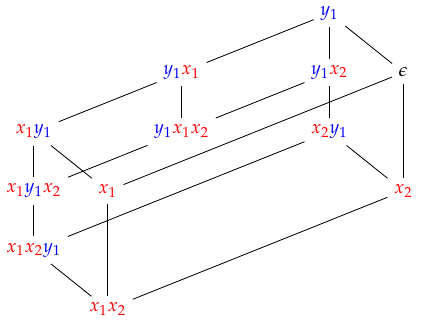}
		\caption{The bubble poset $\Bub(2,1)$.}
		\label{fig:bubble_21}
	\end{subfigure}
	\caption{Two posets of shuffle words.}
	\label{fig:shuffle_bubble}
\end{figure}

\subsection{Basic Properties of Bubble Lattices}

We start by recalling some basic facts of $\Bub(m,n)$ established in \cite{mcconville.muehle:bubbleI}.  For $\ubf\in\Shuf(m,n)$ we define its \defn{inversion set} by
\begin{displaymath}
	\invset(\ubf) \defs \bigl\{(x_{s},y_{t})\colon\;\text{there exist}\;i<j\;\text{such that}\;u_{i}=y_{t}\;\text{and}\;u_{j}=x_{s}\bigr\}.
\end{displaymath}
If $\ubf,\vbf$ are shuffle words, then we denote by $\ubf_{\vbf}$ the restriction of $\ubf$ to the letters appearing in $\vbf$.  Moreover, we consider the words $\xbf \defs x_{1}x_{2}\cdots x_{m}$ and $\ybf\defs y_{1}y_{2}\cdots y_{n}$.  We may now characterize the order relation in $\Bub(m,n)$.  

\begin{lemma}[\cite{mcconville.muehle:bubbleI}*{Lemma~3.1}]\label{lem:bubble_order}
	Let $\ubf,\vbf\in\Shuf(m,n)$.  Then, $\ubf\bubleq\vbf$ if and only if:
	\begin{itemize}
		\item $\vbf_{\xbf}$ is a subword of $\ubf_{\xbf}$,
		\item $\ubf_{\ybf}$ is a subword of $\vbf_{\ybf}$, and 
		\item $\invset(\ubf_{\vbf})\subseteq\invset(\vbf_{\ubf})$.
	\end{itemize}
\end{lemma}

\begin{lemma}[\cite{mcconville.muehle:bubbleI}*{Lemma~3.4}]\label{lem:bubble_duality}
	For $m,n\geq 0$, the lattice $\Bub(m,n)$ is dual to $\Bub(n,m)$.
\end{lemma}

We can also describe the order relation in $\Bub(m,n)$ locally.  That means that we may characterize the \defn{covering pairs} in $\Bub(m,n)$, \ie the pairs $(\ubf,\vbf)$ such that $\ubf\bubless\vbf$ and there exists no $\wbf\in\Shuf(m,n)$ with $\ubf\bubless\wbf\bubless\vbf$.

Let $\ubf=u_{1}u_{2}\cdots u_{k}\in\Shuf(m,n)$.  We denote by $\ubf_{\hat{\i}}$ the word obtained from $\ubf$ by deleting $u_{i}$.  An \defn{indel} is a relation $\ubf\indel\ubf_{\hat{\i}}$ if $u_{i}\in X$ or $\ubf_{\hat{\i}}\indel\ubf$ if $u_{i}\in Y$.  We wish to emphasize a particular type of indel: a \defn{right indel} is defined by
\begin{displaymath}
	\vbf\sindel\vbf' \quad\text{if and only if}\quad \begin{cases}\vbf=\ubf\;\text{and}\;\vbf'=\ubf_{\hat{\i}}, & \text{if}\;u_{i}\in X\;\text{and either}\;u_{i+1}\in X\;\text{or}\;i=k,\\ \vbf=\ubf_{\hat{\i}}\;\text{and}\;\vbf'=\ubf, & \text{if}\;u_{i},u_{i+1}\in Y.\end{cases}.
\end{displaymath}
If $u_{i}\in X$ and $u_{i+1}\in Y$, then a \defn{(forward) transposition} is a transformation $\ubf\transpose\ubf'$, where $\ubf'=u_{1}u_{2}\cdots u_{i-1}u_{i+1}u_{i}u_{i+2}\cdots u_{k}$.  

\begin{lemma}[\cite{mcconville.muehle:bubbleI}*{Lemma~3.5}]\label{lem:bubble_covers}
	For $\ubf,\vbf\in\Shuf(m,n)$ we have $\ubf\bubcov\vbf$ if and only if either $\ubf\transpose\vbf$ or $\ubf\sindel\vbf$.
\end{lemma}

The following lemma states that the graph determined by the covering pairs in $\Bub(m,n)$ is regular of degree $m+n$.

\begin{lemma}[\cite{mcconville.muehle:bubbleI}*{Lemma~3.6}]\label{lem:bubble_hasse_regular}
	For every $\ubf\in\Shuf(m,n)$, the set 
	\begin{displaymath}
		\bigl\{\vbf\in\Shuf(m,n)\colon \vbf\bubcov\ubf\;\text{or}\;\ubf\bubcov\vbf\bigr\}
	\end{displaymath}
	has $m+n$ elements.
\end{lemma}

\subsection{Simplicial Complexes}
    \label{sec:simplicial_complexes}
Some aspects of bubble lattices can be made clearer by considering certain simplicial complexes instead.  We therefore recall some basic notions here, and refer the reader to \cite{ziegler:polytopes} for more background.

Let $M$ be a finite set.  An \defn{(abstract) simplicial complex} on $M$ is a family $\Delta$ of subsets of $M$ such that for $G\in\Delta$ and $F\subseteq G$, it follows that $F\in\Delta$.  The elements of $M$ are the \defn{vertices}, and the elements of $\Delta$ are the \defn{faces}.  The \defn{dimension} of a face $F\in\Delta$ is $\dim(F)\defs\lvert F\rvert-1$.  A \defn{facet} of $\Delta$ is a face $F\in\Delta$ such that for all $G\in\Delta$ with $F\subseteq G$ it follows that $F=G$.  The \defn{dimension} of $\Delta$ is the maximum dimension of its facets.  If all facets of $\Delta$ have the same dimension, then $\Delta$ is \defn{pure}.  If every minimal non-face is of size two, then $\Delta$ is \defn{flag}.

A \defn{ridge} of $\Delta$ is a face of codimension $1$.  Borrowing terminology from the theory of buildings, we call $\Delta$ \defn{thin} if every ridge is contained in exactly two facets.  

The \defn{$f$-vector} of $\Delta$ is the sequence $f(\Delta)\defs(f_{-1},f_{0},f_{1},\ldots)$, where $f_{i}$ denotes the number of $i$-dimensional faces of $\Delta$.  Many properties of $\Delta$ that involve the $f$-vector can be stated more elegantly using the \defn{$h$-vector} instead.  This is the sequence $h(\Delta)\defs(h_{0},h_{1},\ldots)$, given by
\begin{displaymath}
    h_{i} \defs \sum_{k=0}^{i}(-1)^{i-k}\binom{d-k}{d-i}f_{k-1}.
\end{displaymath}
where $d-1$ is the dimension of $\Delta$.  For later use, we encode the $f$- and the $h$-vector associated with $\Delta$ in terms of two polynomials\footnote{Our definition differs from the usual one by reversing the coefficient sequence, see \cite{ziegler:polytopes}*{Section~8.3}.}:
\begin{align}
	f_{\Delta}(q) & \defs \sum_{i\geq 0}f_{i-1}q^{i} = \sum_{\sigma\in\Delta}q^{\lvert \sigma\rvert},\label{eq:f_polynomial}\\
	h_{\Delta}(q) & \defs \sum_{i\geq 0}h_{i}q^{i} = \sum_{i\geq 0}f_{i-1}q^{i}(1-q)^{d-i}\label{eq:h_polynomial}.
\end{align}
It is well known that
\begin{equation}\label{eq:fh_relation}
	f_{\Delta}(q) = (q+1)^{d}h_{\Delta}\left(\frac{q}{q+1}\right),
\end{equation}
which follows for instance from \cite{ziegler:polytopes}*{Section~8.3}.  The \defn{reduced Euler characteristic} of $\Delta$ is
\begin{equation}\label{eq:euler_characteristic}
	\tilde{\chi}(\Delta) \defs -f_{-1} + f_{0} - f_{1} + \cdots + (-1)^{d+1}f_{d} = -f_{\Delta}(-1).
\end{equation}

Of particular interest are simplicial complexes that can be decomposed in a nice and orderly fashion.  For $F\in\Delta$, let $\overline{F}\defs\{G\in\Delta\colon G\subseteq F\}$ be the set of faces of $\Delta$ contained in $F$.  According to \cite{bjorner:shellableI}*{Definition~2.1} a (not necessarily pure) simplicial complex $\Delta$ is \defn{shellable} if there exists a linear order $F_{1},F_{2},\ldots,F_{s}$ of its facets such that $\bigl(\bigcup_{i=1}^{k-1}{\overline{F}_{i}}\bigr)\cap\overline{F}_{k}$ is a pure simplicial complex of dimension $\dim(F_{k})-1$ for all $k\in\{2,3,\ldots,s\}$.  Such an order is a \defn{shelling} of $\Delta$.

Suppose that $\Delta$ is shellable, and let $F_{1},F_{2},\ldots,F_{s}$ be a shelling.  For $j\in[s]$, let $\Res(F_{j})$ denote the \defn{restriction} of $F_{j}$, \ie the minimal face of $F_{j}$ that does not yet appear in the subcomplex generated by $\{F_{i} \colon i<j\}$.  We get the following nice interpretation of the $h$-vector of shellable complexes.

\begin{proposition}[\cite{ziegler:polytopes}*{Theorem~8.19}]\label{prop:h_vector_restriction}
    Let $\Delta$ be a shellable simplicial complex.  Then, $h_{i}$ counts the facets in a shelling of $\Delta$ whose restriction has size $i$.
\end{proposition}

\begin{theorem}[\cite{bjorner:combinatorial}*{Theorem~1.3}]\label{thm:pure_shellable_homotopy}
	Let $\Delta$ be a shellable, pure simplicial complex of dimension $d$.  Then, $\Delta$ is homotopy equivalent to a wedge of $\bigl\lvert\tilde{\chi}(\Delta)\bigr\rvert$-many $d$-dimensional spheres.
\end{theorem}

\medskip

We end this section with the introduction of a property somewhat stronger than shellability following \cites{bjorner:shellableII,provan:decompositions}.  For $F\in\Delta$, the \defn{link} of $F$ in $\Delta$ is
\begin{equation}
	\lnk_{\Delta}(F) \defs \bigl\{G\in\Delta\colon F\cap G=\emptyset\;\text{and}\;F\cup G\in\Delta\bigr\}.
\end{equation}
The \defn{deletion} of $F$ in $\Delta$ is
\begin{equation}
	\dlt_{\Delta}(F) \defs \bigl\{G\in\Delta\colon F\not\subseteq G\bigr\}.
\end{equation}
A simplicial complex is \defn{vertex decomposable} if it is a simplex, if $\Delta=\{\emptyset\}$ or if there exists a vertex $v\in M$ such that
\begin{description}
	\item[VD1\label{it:vd_1}] $\lnk_{\Delta}(v)$ is vertex decomposable;
	\item[VD2\label{it:vd_2}] $\dlt_{\Delta}(v)$ is vertex decomposable;
	\item[VD3\label{it:vd_3}] the complexes $\lnk_{\Delta}(v)$ and $\dlt_{\Delta}(v)$ do not share facets.
\end{description}
The distinguished vertex $v$ is a \defn{shedding vertex} of $\Delta$.

\begin{proposition}[\cite{bjorner:shellableII}*{Theorem~11.3}]\label{prop:vd_implies_shellable}
	Every vertex-decomposable simplicial complex is shellable.
\end{proposition}

For two simplicial complexes $\Delta_{1}$ and $\Delta_{2}$ with disjoint vertex sets we define its \defn{join} by
\begin{equation}
	\Delta_{1}*\Delta_{2} \defs \bigl\{F_{1}\uplus F_{2}\colon F_{1}\in\Delta_{1}\;\text{and}\;F_{2}\in\Delta_{2}\bigr\}.
\end{equation}

\begin{proposition}[\cite{jonsson:simplicial}*{Theorem~3.30}]\label{prop:vd_join}
	The join of two vertex-decomposable complexes is again vertex decomposable.
\end{proposition}

\section{The Noncrossing Matching Complex}\label{sec:match}

\subsection{Definition}

We start our combinatorial study of the bubble lattices with the definition of a simplicial complex, whose faces turn out to be in bijection with the elements of $\Bub(m,n)$.  
Recall that $X=\{x_1,x_{2},\ldots,x_m\}$ and $Y=\{y_1,y_{2},\ldots,y_n\}$, and define
\begin{equation}\label{eq:edges}
	\Ecal(X,Y) \defs \bigl\{\{x,y\}\colon x\in X, y\in Y\bigr\}.
\end{equation}
Then, we consider the set
\begin{equation}\label{eq:bubble_labels}
	\Tcal \defs X\uplus Y\uplus\Ecal(X,Y).
\end{equation}
For convenience, we call the elements of $X\uplus Y$ the \defn{loops}, while elements of $\Ecal(X,Y)$ are \defn{edges}.

As explained in \cite{mcconville.muehle:bubbleI}*{Section~4.3}, the covering pairs of $\Bub(m,n)$ are naturally labeled by $\Tcal$.  More precisely, if $\ubf\bubcov\vbf$, then
\begin{equation}\label{eq:bubble_labeling}
	\lambda(\ubf,\vbf) \defs 
	\begin{cases}
		x, & \text{if}\;\ubf\sindel\vbf, \ubf_{\ybf}=\vbf_{\ybf},\ubf_{\xbf}\setminus\vbf_{\xbf}=\{x\},\\
		y, & \text{if}\;\ubf\sindel\vbf, \ubf_{\xbf}=\vbf_{\xbf},\vbf_{\ybf}\setminus\ubf_{\ybf}=\{y\},\\
		\{x,y\}, & \text{if}\;\ubf\transpose\vbf, \invset(\vbf)\setminus\invset(\ubf)=\bigl\{\{x,y\}\bigr\}.
	\end{cases}
\end{equation}
Figure~\ref{fig:bubble_21_labeled} shows $\Bub(2,1)$ with this labeling.  Given $\vbf\in\Shuf(m,n)$, we define the set of \defn{downward labels} by
\begin{equation}
	\lambda_{\downarrow}(\vbf) \defs \bigl\{\lambda(\ubf,\vbf)\colon \ubf\bubcov\vbf\bigr\}.
\end{equation}

\begin{figure}
	\centering
	\begin{subfigure}[t]{.45\textwidth}
		\includegraphics[page=6,scale=1]{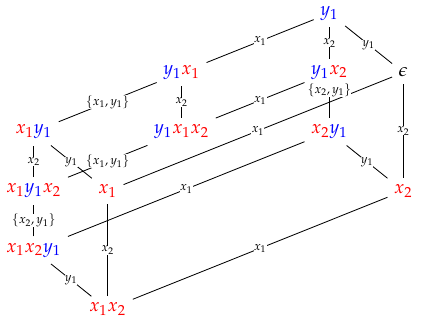}
		\caption{The poset $\Bub(2,1)$ labeled by $\lambda$.}
		\label{fig:bubble_21_labeled}
	\end{subfigure}
	\hspace*{1cm}
	\begin{subfigure}[t]{.45\textwidth}
		\centering
		\includegraphics[page=10,scale=1]{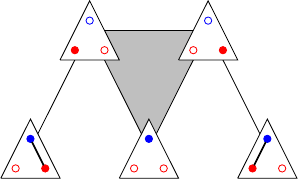}
		\caption{The noncrossing matching complex $\Gamma(2,1)$.}
		\label{fig:ncm_21}
	\end{subfigure}
	\caption{Illustration of Proposition~\ref{prop:ncm_faces}.  In the vertices of $\Gamma(2,1)$, the two red nodes on the bottom correspond to $x_{1}$ and $x_{2}$ (from left to right), while the top blue node corresponds to $y_{1}$.  The gray triangle represents a two-dimensional face.}
\end{figure}

The \defn{noncrossing matching complex} $\Gamma\defs\Gamma(m,n)$ is the abstract simplicial complex with vertex set $\Tcal$ whose faces $\sigma$ satisfy the following conditions:
\begin{description}
	\item[NM1\label{it:nm_1}] each letter $x_s$ or $y_t$ appears at most once in $\sigma$, \quad and
	\item[NM2\label{it:nm_2}] if $\{x_{s_1},y_{t_1}\}$ and $\{x_{s_2},y_{t_2}\}$ are in $\sigma$ such that $s_1<s_2$ then $t_1<t_2$.
\end{description}
Figure~\ref{fig:ncm_21} shows $\Gamma(2,1)$ and Figure~\ref{fig:ncm_22} shows $\Gamma(2,2)$.

\begin{figure}
	\centering
	\includegraphics[page=25,scale=1]{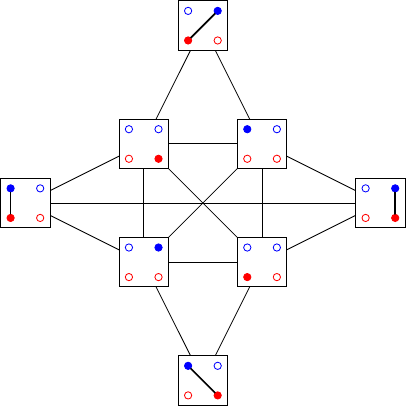}
	\caption{The graph of the noncrossing matching complex $\Gamma(2,2)$.  Since $\Gamma(2,2)$ is flag, this graph determines the whole complex.}
	\label{fig:ncm_22}
\end{figure}

\begin{proposition}\label{prop:ncm_faces}
  For $\sigma\subseteq\Tcal$, the set $\sigma$ is a face of $\Gamma$ if and only if there exists $\vbf_{\sigma}\in\Shuf(m,n)$ such that $\sigma=\lambda_{\downarrow}(\vbf_{\sigma})$.
\end{proposition}
\begin{proof}
	First consider $\vbf\in\Shuf(m,n)$ and let $\sigma=\lambda_{\downarrow}(\vbf)$.  For any $\ell\in\sigma$, there exists some $\ubf\bubcov\vbf$ such that $\lambda(\ubf,\vbf)=\ell$.  
	
	If $\ell=x$, $\vbf$ is obtained from $\ubf$ by a right indel deleting $x$, which means that $x\notin\vbf$.  If $\ell=y$, then $\vbf$ is obtained from $\ubf$ by a right indel inserting $y$.  In particular, either $y$ is the last letter of $\vbf$ or the letter immediately after $y$ in $\vbf$ is in $Y$.  If $\ell=\{x,y\}$, then $\vbf$ is obtained from $\ubf$ by a transposition swapping $x$ and $y$.  In particular, $x\in\vbf$ and the letter immediately after $y$ in $\vbf$ is $x$.  This means that each element of $X$ and $Y$ appears at most once; thus $\sigma$ satisfies \eqref{it:nm_1}.  Moreover, if $\{x_{s_{1}},y_{t_{1}}\},\{x_{s_{2}},y_{t_{2}}\}\in\sigma$, then $y_{t_{1}}$ immediately precedes $x_{s_{1}}$ in $\vbf$ and $y_{t_{2}}$ immediately precedes $x_{s_{2}}$ in $\vbf$.  If $s_{1}<s_{2}$, then $x_{s_{1}}$ precedes $x_{s_{2}}$, because $\vbf$ is a shuffle word meaning that $\vbf_{\xbf}$ is a subword of $\xbf$.  But then, $y_{t_{1}}$ must precede $y_{t_{2}}$, which forces $t_{1}<t_{2}$, because $\vbf_{\ybf}$ is a subword of $\ybf$.  This establishes \eqref{it:nm_2}.  Thus, $\sigma\in\Gamma$.
	
	\medskip
	
	Conversely, pick a face $\sigma$ of $\Gamma$.  Let $\{x_{i_{1}},x_{i_{2}},\ldots,x_{i_{s}}\}=X\setminus\sigma$ such that $i_{1}<i_{2}<\cdots<i_{s}$.  Let $\ubf^{(0)}\defs x_{i_{1}}x_{i_{2}}\cdots x_{i_{s}}$.  For any $\{x,y\}\in\sigma$, \eqref{it:nm_1} guarantees that $x=x_{i_{k}}$ for some $k\in[s]$.  We insert $y$ in $\ubf^{(0)}$ immediately left of $x_{i_{k}}$.  Let $\vbf^{(0)}$ denote the word obtained after all edges in $\sigma$ have been dealt with in this manner.  Condition~\eqref{it:nm_2} ensures that $\vbf^{(0)}_{\ybf}$ is a subword of $\ybf$.  Thus, $\vbf^{(0)}\in\Shuf(m,n)$.  If $Y\cap\sigma=\{y_{j_{1}},y_{j_{2}},\ldots,y_{j_{t}}\}$ with $j_{1}<j_{2}<\cdots<j_{t}$, then \eqref{it:nm_1} ensures that $y_{j_{k}}\notin\ubf^{(1)}$.  We construct $\vbf^{(k)}$ from $\vbf^{(k-1)}$ by inserting $y_{j_{k}}$ either at the end (if $\vbf^{(k-1)}$ does not contain a letter $y_{r}$ with $r>j_{k}$) or immediately left of the letter $y_{r}$ where $r$ is the smallest index greater than $j_{k}$ such that $y_{r}\in\vbf^{(k-1)}$.  Then, we set $\vbf_{\sigma}\defs\vbf^{(t)}$.  By construction, $\vbf_{\sigma}\in\Shuf(m,n)$ and $\lambda_{\downarrow}(\vbf_{\sigma})=\sigma$.
\end{proof}

\begin{example}
	Let $m=7$ and $n=5$, and consider $\vbf=\bl{y_2}\rd{x_1x_4}\bl{y_3y_{4}y_5}\rd{x_5x_6}$.  We have 
	\begin{displaymath}
		\lambda_{\downarrow}(\vbf) = \bigl\{\rd{x_2},\rd{x_3},\rd{x_7},\bl{y_3},\bl{y_4},\{\rd{x_1},\bl{y_2}\},\{\rd{x_5},\bl{y_5}\}\bigr\}\in\Gamma(7,5).
	\end{displaymath}
	
	Conversely, if $\sigma=\bigl\{\rd{x_2},\rd{x_3},\rd{x_7},\bl{y_3},\bl{y_4},\{\rd{x_1},\bl{y_2}\},\{\rd{x_5},\bl{y_5}\}\bigr\}\in\Gamma(7,5)$, then we get $X\setminus\sigma=\{\rd{x_1},\rd{x_4},\rd{x_5},\rd{x_6}\}$ and thus $\ubf^{(0)}=\rd{x_1x_4x_5x_6}$.  Dealing with $\{\rd{x_1},\bl{y_2}\}$ and $\{\rd{x_5},\bl{y_5}\}$ in that order produces $\vbf^{(0)}=\bl{y_2}\rd{x_1x_4}\bl{y_5}\rd{x_5x_6}$.  Finally, $Y\cap\sigma=\{\bl{y_3},\bl{y_4}\}$, and we get
	\begin{align*}
		\vbf^{(1)} & = \bl{y_2}\rd{x_1x_4}\bl{y_3y_5}\rd{x_5x_6},\\
		\vbf^{(2)} & = \bl{y_2}\rd{x_1x_4}\bl{y_3y_4y_5}\rd{x_5x_6}.
	\end{align*}
\end{example}

Since any $\vbf\in\Shuf(m,n)$ is uniquely determined by its support and its inversion set, it is clear that the assignment $\vbf\mapsto\lambda_{\downarrow}(\vbf)$ is a bijection from $\Shuf(m,n)$ to (the set of faces of) $\Gamma(m,n)$.

\begin{remark}\label{rem:bubble_join_complex}
	It was shown in \cite{mcconville.muehle:bubbleI}*{Theorem~1.1} that $\Bub(m,n)$ is semidistributive.  In general, the elements of a semidistributive lattice $\Lattice$ can be described by means of \emph{canonical join representations}, and the set of all canonical join representations forms a simplicial complex; the \emph{canonical join complex} of $\Lattice$, see for instance \cite{barnard:canonical}.  By \cite{barnard:canonical}*{Lemma~19}, the canonical join complex of $\Lattice$ can be described using a specific labeling of the covering pairs of $\Lattice$.
	
	In the case of the bubble lattice, \cite{mcconville.muehle:bubbleI}*{Proposition~4.15} implies together with Proposition~\ref{prop:ncm_faces} that $\Gamma(m,n)$ is isomorphic to the canonical join complex of $\Bub(m,n)$.
\end{remark}

Remark~\ref{rem:bubble_join_complex} together with \cite{barnard:canonical}*{Theorem~1} implies the following property of $\Gamma(m,n)$.

\begin{corollary}\label{cor:ncm_flag}
	For $m,n\geq 0$, the noncrossing matching complex $\Gamma(m,n)$ is flag.
\end{corollary}

\subsection{Face Enumeration}

We may use Proposition~\ref{prop:ncm_faces} to enumerate the faces of $\Gamma(m,n)$.  One immediate consequence of Proposition~\ref{prop:ncm_faces} is that the set of $i$-dimensional faces of $\Gamma(m,n)$ is in bijection with the set
\begin{displaymath}
	\Bigl\{\vbf\in\Shuf(m,n)\colon \bigl\lvert\bigl\{\ubf\in\Shuf(m,n)\colon \ubf\bubcov\vbf\bigr\}\bigr\rvert=i+1\Bigr\}.
\end{displaymath}
Let us therefore define the \defn{in-degree} of $\vbf\in\Shuf(m,n)$ by
\begin{displaymath}
	\indeg(\vbf) \defs \bigl\lvert\bigl\{\ubf\in\Shuf(m,n)\colon \ubf\bubcov\vbf\bigr\}\bigr\rvert.
\end{displaymath}
In view of Lemma~\ref{lem:bubble_covers}, there are essentially two types of covering pairs in $\Bub(m,n)$.  This dichotomy may be realized by defining the \defn{indel-degree} of $\ubf$ by
\begin{displaymath}
	\indeld(\ubf) \defs \bigl\lvert\bigl\{\ubf'\in\Shuf(m,n)\colon \ubf'\sindel\ubf\bigr\}\bigr\rvert,
\end{displaymath}
and the \defn{transpose-degree} of $\ubf$ by
\begin{displaymath}
  \transd(\ubf) \defs \bigl\lvert\bigl\{\ubf'\in\Shuf(m,n)\colon \ubf'\transpose\ubf\bigr\}\bigr\rvert.
\end{displaymath}
Evidently, $\indeg(\ubf) = \indeld(\ubf) + \transd(\ubf)$.

\begin{lemma}\label{lem:bubble_refined_indegree}
  The number of $\ubf\in\Shuf(m,n)$ with $\transd(\ubf)=a$ and $\indeld(\ubf)=b$ is $\binom{m}{a}\binom{n}{a}\binom{m+n-2a}{b}$.
\end{lemma}
\begin{proof}
	It is straightforward to see that the shuffle word $\ubf$ is completely determined by its interface and its residue. The statistic $\transd(\ubf)$ is the number of consecutive subwords $yx$ with $x\in X$ and $y\in Y$. So $\transd(\ubf)$ is half of the size of the interface.

	Any $x\in X$ not in the support of $\ubf$ can be inserted to give a relation $\ubf'\sindel \ubf$. Conversely, any $y\in Y$ in the residue of $\ubf$ can be deleted to give a relation $\ubf'' \sindel \ubf$. Consequently, $\indeld(\ubf)$ is equal to the number of elements of $Y$ in the residue of $\ubf$ plus the number of elements of $X$ not in the support of $\ubf$.

	We now construct all words $\ubf$ with $\transd(\ubf)=a$ and $\indeld(\ubf)=b$. Let $I$ be the union of a subset of $X$ and a subset of $Y$, each of size $a$. Given this set, we choose a $b$-element subset $Z$ of the $m+n-2a$ elements of $(X\cup Y)\setm I$ and set $R=(X\setm (Z\cup I)) \cup (Y\cap Z)$. The word $\ubf$ with interface $I$ and residue $R$ satisfies $\transd(\ubf)=a$ and $\indeld(\ubf)=b$. The number of ways to construct such an interface and residue is $\displaystyle \binom{m}{a}\binom{n}{a}\binom{m+n-2a}{b}$.
\end{proof}

Let us consider the following polynomial, the \defn{$H$-triangle}, which is a refined variant of the $f$-polynomial of $\Gamma(m,n)$.
\begin{equation}\label{eq:bubble_h_triangle}
	H_{m,n}(q,t) \defs \sum_{\ubf\in\Shuf(m,n)}q^{\indeg(\ubf)}t^{\indeld(\ubf)}.
\end{equation}

\begin{lemma}\label{lem:ncm_face_polynomial}
	For $m,n\geq 0$, we have $f_{\Gamma(m,n)}(q) = H_{m,n}(q,1)$.
\end{lemma}
\begin{proof}
	By Proposition~\ref{prop:ncm_faces}, we immediately get
	\begin{displaymath}
		H_{m,n}(q,1) = \sum_{\ubf\in\Shuf(m,n)}q^{\indeg(\ubf)} = \sum_{\sigma\in\Gamma(m,n)}q^{\lvert\sigma\rvert} = f_{\Gamma(m,n)}(q).\qedhere
	\end{displaymath}
\end{proof}

Lemma~\ref{lem:bubble_refined_indegree} enables us to write down an explicit formula for $H_{m,n}(q,t)$, which is our first main result that we repeat here for the convenience of the reader.

\bubblehtriangle*
\begin{proof}
    Let $\ubf\in\Shuf(m,n)$.  If $\transd(\ubf)=a$, then the reasoning from the proof of Lemma~\ref{lem:bubble_refined_indegree} implies that $\indeld(\ubf)\leq m+n-2a$.  By using Lemma~\ref{lem:bubble_refined_indegree} and the Binomial Theorem we get
    \begin{align*}
        H_{m,n}(q,t) & = \sum_{\ubf\in\Shuf(m,n)}q^{\indeg(\ubf)}t^{\indeld(\ubf)}\\
        & = \sum_{\ubf\in\Shuf(m,n)}q^{\transd(\ubf)}(qt)^{\indeld(\ubf)}\\
        & = \sum_{a=0}^{\min\{m,n\}}\sum_{b=0}^{m+n-2a}\binom{m}{a}\binom{n}{a}\binom{m+n-2a}{b}q^{a}(qt)^{b}\\
        & = \sum_{a=0}^{\min\{m,n\}}\binom{m}{a}\binom{n}{a}q^{a}(qt+1)^{m+n-2a}.\qedhere
    \end{align*}
\end{proof}

\begin{corollary}\label{cor:ncm_euler}
	For $m,n\geq 0$, the reduced Euler characteristic of $\Gamma(m,n)$ is
	\begin{displaymath}
		\tilde{\chi}\bigl(\Gamma(m,n)\bigr) = \begin{cases}(-1)^{n}, & \text{if}\;m=n,\\0, & \text{otherwise}.\end{cases}
	\end{displaymath}
\end{corollary}
\begin{proof}
	By combining \eqref{eq:euler_characteristic} with Lemma~\ref{lem:ncm_face_polynomial}, we obtain
	\begin{displaymath}
		\tilde{\chi}\bigl(\Gamma(m,n)\bigr) = -f_{\Gamma(m,n)}(-1) = H_{m,n}(-1,1) = \sum_{a=0}^{\min\{m,n\}}\binom{m}{a}\binom{n}{a}(-1)^{a}0^{m+n-2a}.
	\end{displaymath}
	The last sum consists only of the summand corresponding to $a=\frac{m+n}{2}$ if it exists.  If $m$ and $n$ have different parity, then this summand does not exist.  
	
	If not, by symmetry, we may assume that $m\geq n$.  If $m=n$, this summand is $(-1)^{n}$.  So, suppose that $m>n$.  If $m=2k$ and $n=2l$, then the corresponding summand is $(-1)^{k+l}\binom{2k}{k+l}\binom{2l}{k+l}$.  However, $m>n$ implies $k>l$ and thus $\binom{2l}{k+l}=0$.  If $m=2k+1$ and $n=2l+1$, then the corresponding summand is $(-1)^{k+l+1}\binom{2k+1}{k+l+1}\binom{2l+1}{k+l+1}$.  Once again, since $k>l$ we get $\binom{2l+1}{k+l+1}=0$.
\end{proof}

\begin{remark}
	Somewhat surprisingly, the polynomial $H_{m,n}(q,1)$ appears already in \cite{greene:shuffle}*{Corollary~4.8} as the rank-generating polynomial of $\ShufPoset(m,n)$, and was used to establish the rank symmetry of $\ShufPoset(m,n)$ and for constructing a decomposition into symmetrically placed Boolean lattices.
\end{remark}

\begin{remark}\label{rem:fh_correspondence}
	We may define an labeling $\gamma$ of the covering pairs in $\Bub(m,n)$ by setting
	\begin{displaymath}
		\gamma(\ubf,\vbf) \defs \begin{cases}0, & \text{if}\;\ubf\transpose\vbf,\\1, & \text{if}\;\ubf\sindel\vbf.\end{cases}
	\end{displaymath}
	Then, the $H$-triangle defined in \eqref{eq:bubble_h_triangle} agrees with the bivariate $H$-triangle of $\Bub(m,n)$ with respect to $\gamma$ as defined in \cite{ceballos:fh}*{Section~5}.  By \cite{ceballos:fh}*{Theorem~5.1}, there is a naturally associated integer polynomial
	\begin{displaymath}
		F_{m,n}(q,t) = q^{m+n}H_{m,n}\left(\frac{q+1}{q},\frac{t+1}{q+1}\right).
	\end{displaymath}
	We give a geometric interpretation of this polynomial in Section~\ref{sec:delta_enumeration}.  An analogous relationship has been exposed by Chapoton for the Tamari lattices in \cite{chapoton:enumerative} and by Garver and the first author for the Grid--Tamari order~\cite{garver.mcconville:triangles}.  
\end{remark}

\subsection{Vertex-Decomposability}
	\label{sec:ncm_vertex_decomposable}

We now prove that $\Gamma(m,n)$ has another intriguing topological property: we show that it is (non-pure) vertex decomposable in the sense of \cites{bjorner:shellableII,provan:decompositions}.

\begin{theorem}\label{thm:ncm_vertex_decomposable}
	For $m,n\geq 0$, the noncrossing matching complex $\Gamma(m,n)$ is vertex decomposable.
\end{theorem}

We start with some small cases.  

\begin{proposition}\label{prop:ncm_small}
	For $n\leq 1$ and $m\geq 0$, the noncrossing matching complex $\Gamma(m,n)$ is vertex decomposable.
\end{proposition}
\begin{proof}
	The noncrossing matching complex $\Gamma(m,0)$ is a simplex and therefore vertex decomposable by definition.  The noncrossing matching complex $\Gamma(m,1)$ is isomorphic to the canonical join complex of the Hochschild lattice $\Hoch(n)$ by \cite{mcconville.muehle:bubbleI}*{Proposition~4.19} and Remark~\ref{rem:bubble_join_complex}.  It was shown in \cite{muehle:hochschild}*{Theorem~1.2} that this complex is vertex decomposable.
\end{proof}

It follows immediately from the definition that $\Gamma(m,n)\cong\Gamma(n,m)$.  In the remainder of this section, we may therefore assume $\min\{m,n\}\geq 2$.  

\begin{lemma}\label{lem:ncm_links}
	For $s\in[m]$ and $t\in[n]$, 
	\begin{displaymath}
		\lnk_{\Gamma(m,n)}\bigl(\{x_{s},y_{t}\}\bigr) \cong \Gamma(s-1,t-1)*\Gamma(m-s,n-t).
	\end{displaymath}
\end{lemma}
\begin{proof}
	Let us write $\Delta_{s,t}\defs\lnk_{\Gamma(m,n)}\bigl(\{x_{s},y_{t}\}\bigr)$, and consider $F\in\Delta_{s,t}$.  By definition, $F\cup\bigl\{\{x_{s},y_{t}\}\bigr\}\in\Gamma(m,n)$.  Thus, if $x_{s'}\in F$, then by \eqref{it:nm_1}, $s'\neq s$.  Likewise, if $y_{t'}\in F$, then by \eqref{it:nm_1}, $t'\neq t$.  If $\{x_{s'},y_{t'}\}\in F$, then by \eqref{it:nm_1} $s'\neq s$ and $t'\neq t$ and by \eqref{it:nm_2}, either $s'>s$ and $t'>t$ or $s'<s$ and $t'<t$.  
	
	Let $X_{1}=\{x_{1},x_{2},\ldots,x_{s-1}\}$, $X_{2}=\{x_{s+1},x_{s+2},\ldots,x_{m}\}$, $Y_{1}=\{y_{1},y_{2},\ldots,y_{t-1}\}$, $Y_{2}=\{y_{t+1},y_{t+2},\ldots,y_{n}\}$.  Let 
	\begin{align*}
		F_{1} & = F\cap\bigl(X_{1}\uplus Y_{1}\uplus\Ecal(X_{1},Y_{1})\bigr)\\
		F_{2} & = F\cap\bigl(X_{2}\uplus Y_{2}\uplus\Ecal(X_{2},Y_{2})\bigr).
	\end{align*}
	Then $F=F_{1}\uplus F_{2}$ by the reasoning from the first paragraph.  Moreover $F_{1}\in\Gamma(s-1,t-1)$ and for $F_{2}$ there exists a unique face $\tilde{F}_{2}$ of $\Gamma(m-s,n-t)$ which is obtained by decreasing each subscript of the $x$'s in $F_{2}$ by $s$ and each subscript by the $y$'s in $F_{2}$ by $t$.  Thus, $F$ corresponds to a unique face of $\Gamma(s-1,t-1)*\Gamma(m-s,n-t)$.
	
	\medskip
	
	Conversely, let $F_{1}\in\Gamma(s-1,t-1)$ and $\tilde{F}_{2}\in\Gamma(m-s,n-t)$ and construct $F_{2}$ by increasing each subscript of the $x$'s by $s$ and each subscript of the $y$'s by $t$.  Then, $F=F_{1}\uplus F_{2}\in\Gamma(m,n)$.  Moreover, by construction, no vertex of $F$ uses the letters $x_{s}$ or $y_{t}$ and if $\{x_{s'},y_{t'}\}\in F$, then either $s'<s$ and $t'<t$ or $s'>s$ and $t'>t$.  Therefore, $F\cup\bigl\{\{x_{s},y_{t}\}\bigr\}\in\Gamma(m,n)$.  This proves that $F\in\Delta_{s,t}$.  
\end{proof}

Let $\Delta_{0}\defs\Gamma(m,n)$, and for $t\in[n]$ define $\Delta_{t}\defs\dlt_{\Delta_{t-1}}\bigl(\{x_{1},y_{t}\}\bigr)$.  Then, $\Delta_{t}$ is a subcomplex of $\Gamma(m,n)$ so every face of $\Delta_{t}$ must satisfy \eqref{it:nm_1} and \eqref{it:nm_2}.

\begin{lemma}\label{lem:dn_join}
	We have $\Delta_{n}\cong\{x_{1}\}*\Gamma(m-1,n)$.  
\end{lemma}
\begin{proof}
	To obtain $\Delta_{n}$, we delete all vertices $\{x_{1},y_{1}\},\{x_{1},y_{2}\},\ldots,\{x_{1},y_{n}\}$ from $\Gamma(m,n)$.  Let $F\in\Delta_{n}$.  Then $F\setminus\{x_{1}\}$ corresponds to a unique face $\tilde{F}\in\Gamma(m-1,n)$ by decreasing the indices of the $x$'s in $F\setminus\{x_{1}\}$ by one.
	
	Conversely, every face $\tilde{F}\in\Gamma(m-1,n)$ corresponds to a unique face $F\in\Delta_{n}$ by increasing the indices of the $x$'s in $\tilde{F}$ by one.  Clearly, $F\cup\{x_{1}\}$ is noncrossing and belongs to $\Delta_{n}$.  This proves the claim.
\end{proof}

\begin{lemma}\label{lem:dt_links}
	For $t\in[n]$, $\lnk_{\Delta_{t-1}}\bigl(\{x_{1},y_{t}\}\bigr)=\lnk_{\Gamma(m,n)}\bigl(\{x_{1},y_{t}\}\bigr)$.
\end{lemma}
\begin{proof}
	Let $F\in\lnk_{\Delta_{t-1}}\bigl(\{x_{1},y_{t}\}\bigr)$.  By definition, $F\cap\bigl\{\{x_{1},y_{t}\}\bigr\}=\emptyset$ and $F\uplus\{\{x_{1},y_{t}\}\}\in\Delta_{t-1}$.  Since $\Delta_{t-1}$ is a subcomplex of $\Gamma(m,n)$ obtained by deleting some vertices, we conclude that $F\uplus\bigl\{\{x_{1},y_{t}\}\bigr\}\in\Gamma(m,n)$, and thus $F\in\lnk_{\Gamma(m,n)}\bigl(\{x_{1},y_{t}\}\bigr)$.
	
	Conversely, let $F\in\lnk_{\Gamma(m,n)}\bigl(\{x_{1},y_{t}\}\bigr)$.  By definition $F\cap\bigl\{\{x_{1},y_{t}\}\bigr\}=\emptyset$ and $F\uplus\bigl\{\{x_{1},y_{t}\}\bigr\}\in\Gamma(m,n)$.  By \eqref{it:nm_2}, we see that $\{x_{1},y_{t'}\}\notin F$ for all $t'\in[n]$, which implies that $F\in\Delta_{t-1}$.  We get that $F\uplus\bigl\{\{x_{1},y_{t}\}\bigr\}\in\Delta_{t-1}$, and thus $F\in\lnk_{\Delta_{t-1}}\bigl(\{x_{1},y_{t}\}\bigr)$.  
\end{proof}

\begin{lemma}\label{lem:dt_vd3}
	For $t\in[n]$, the complexes $\lnk_{\Delta_{t-1}}\bigl((x_{1},y_{t})\bigr)$ and $\Delta_{t}$ do not share facets.
\end{lemma}
\begin{proof}
	Let $F\in\lnk_{\Delta_{t-1}}\bigl(\{x_{1},y_{t}\}\bigr)$ be a facet.  Since $F\uplus\bigl\{\{x_{1},y_{t}\}\bigr\}\in\Delta_{t-1}$ we conclude from \eqref{it:nm_1} and \eqref{it:nm_2} that $\{x_{1},y_{t'}\}\notin F$ for all $t'\in[n]$ and $x_{1}\notin F$.  But this means that $F\uplus\{x_{1}\}\in \Delta_{t}$.  Consequently, $F$ is not a facet of $\Delta_{t}$.
	
	Conversely, let $F\in\Delta_{t}$ be a facet.  Assume that $F\in\lnk_{\Delta_{t-1}}\bigl(\{x_{1},y_{t}\}\bigr)$.  As before, we conclude that $\{x_{1},y_{t'}\}\notin F$ for all $t'\in[n]$ and $x_{1}\notin F$, because $F\uplus\bigl\{\{x_{1},y_{t}\}\bigr\}\in\Delta_{t-1}$.  But this means that $F\uplus\{x_{1}\}$ satisfies \eqref{it:nm_1} and \eqref{it:nm_2} and consists of vertices of $\Delta_{t}$.  Consequently, $F\uplus\{x_{1}\}\in\Delta_{t}$, contradicting the assumption that $F$ is a facet of $\Delta_{t}$.  It follows that $F\notin\lnk_{\Delta_{t-1}}\bigl(\{x_{1},y_{t}\}\bigr)$; in particular $F$ is not a facet of $\lnk_{\Delta_{t-1}}\bigl(\{x_{1},y_{t}\}\bigr)$.
\end{proof}

We have now gathered all ingredients to prove Theorem~\ref{thm:ncm_vertex_decomposable}.  

\begin{proof}[Proof of Theorem~\ref{thm:ncm_vertex_decomposable}]
	We prove the claim by induction on $m$.  If $m<2$, then the claim follows from Proposition~\ref{prop:ncm_small}.  Now let $m\geq 2$ and suppose that for every $n$ the complex $\Gamma(m-1,n)$ is vertex decomposable.  
	
	We now argue that for $t\in[n]$ the vertex $\{x_{1},y_{t}\}$ is a shedding vertex of $\Delta_{t-1}$.  Lemma~\ref{lem:dt_vd3} proves that $\{x_{1},y_{t}\}$ satisfies \eqref{it:vd_3}.  Lemmas~\ref{lem:ncm_links} and \ref{lem:dt_links} imply by induction that $\{x_{1},y_{t}\}$ satisfies \eqref{it:vd_1}.  Then, as long as $t<n$, $\{x_{1},y_{t}\}$ is a shedding vertex of $\Delta_{t-1}$ if and only if $\{x_{1},y_{t+1}\}$ is a shedding vertex of $\Delta_{t}$.  It remains to check that $\Delta_{n}$ is vertex decomposable, which follows by induction from Lemma~\ref{lem:dn_join}.
\end{proof}

\begin{remark}
	We have chosen the shedding vertex $\{x_{1},y_{1}\}$ of $\Gamma(m,n)$ only for convenience.  Our line of reasoning carries over to \emph{any} vertex $\{x_{s},y_{t}\}$.  
\end{remark}

\begin{corollary}\label{cor:ncm_shellable}
	For $m,n\geq 0$, the noncrossing matching complex $\Gamma(m,n)$ is shellable.
\end{corollary}
\begin{proof}
	Theorem~\ref{thm:ncm_vertex_decomposable} and Proposition~\ref{prop:vd_implies_shellable} imply that $\Gamma(m,n)$ is shellable. 
\end{proof}

In order to compute the exact homotopy type of $\Gamma(m,n)$, we use a refined face enumeration of $\Gamma(m,n)$ that was developed in \cite{bjorner:shellableI}.  Given a simplicial complex $\Delta$ and a face $F\in\Delta$, we define its \defn{degree} by
\begin{displaymath}
	\delta(F) \defs \max\bigl\{\lvert G\rvert\colon F\subseteq G\in\Delta\bigr\}.
\end{displaymath}
Then, $f_{i,j}$ counts the faces of $\Delta$ with degree $i$ and cardinality $j$.  The \defn{BW-$F$-triangle} is
\begin{displaymath}
	F_{\Delta}^{BW}(q,t) \defs \sum_{F\in\Delta} q^{\delta(F)}t^{\delta(F)-\lvert F\rvert} = \sum_{0\leq j\leq i}f_{i,j}q^{i}t^{i-j}.
\end{displaymath}
As with the usual $f$-numbers, we can define a $h$-version of the refined $f$-numbers from above, by setting:
\begin{displaymath}
	h_{i,j} \defs \sum_{k=0}^{j}(-1)^{j-k}\binom{i-k}{j-k}f_{i,k}.
\end{displaymath}
The resulting generating function is the \defn{BW-$H$-triangle}:
\begin{displaymath}
	H_{\Delta}^{BW}(q,t) \defs \sum_{0\leq j\leq i}h_{i,j}q^{i}t^{i-j} = F_{\Delta}^{BW}(q,t-1).
\end{displaymath}
The main purpose of these definitions is the following result.  

\begin{theorem}[\cite{bjorner:shellableI}*{Theorem~4.1}]\label{thm:nonpure_shellable_homotopy}
	Let $\Delta$ be a shellable simplicial complex of dimension $d-1$.  Then, $\Delta$ has the homotopy type of a wedge of spheres, consisting of $h_{j,j}$ copies of the $(j-1)$-dimensional sphere for $j\in[d]$.
\end{theorem}

Let us compute the relevant polynomials for $\Gamma(m,n)$.  

\begin{proposition}\label{prop:gamma_bw_ftriangle}
	For $m,n\geq 0$, we have
	\begin{displaymath}
		F_{\Gamma(m,n)}^{BW}(q,t) = q^{m+n}H_{m,n}\left(\frac{1}{q},qt\right).
	\end{displaymath}
\end{proposition}
\begin{proof}
	Let $F\in\Gamma(m,n)$ and suppose that $F$ has $a$ edges and $b$ loops.  Then, $F$ can be extended to a facet in a maximal way by adding all the $m+n-2a-b$ missing loops.  Therefore, $\delta(F) = m+n-a$ and $\lvert F\rvert=a+b$.  A face of this form can, clearly, be chosen in $\binom{m}{a}\binom{n}{a}\binom{m+n-2a}{b}$ ways.  This gives
	\begin{align*}
		F_{\Gamma(m,n)}^{BW}(q,t) & = \sum_{a=0}^{\min\{m,n\}}\sum_{b=0}^{m+n-2a}\binom{m}{a}\binom{n}{a}\binom{m+n-2a}{b}q^{m+n-a}t^{m+n-2a-b}\\
		& = q^{m+n}\sum_{a=0}^{\min\{m,n\}}\binom{m}{a}\binom{n}{a}q^{-a}\sum_{b=0}^{m+n-2a}\binom{m+n-2a}{b}t^{m+n-2a-b}\\
		& = q^{m+n}\sum_{a=0}^{\min\{m,n\}}\binom{m}{a}\binom{n}{a}q^{-a}(t+1)^{m+n-2a}\\
		& = q^{m+n}H_{m,n}\left(\frac{1}{q},qt\right).\qedhere
	\end{align*}
\end{proof}

\begin{corollary}\label{cor:gamma_bw_htriangle}
	For $m,n\geq 0$, we have
	\begin{displaymath}
		H_{\Gamma(m,n)}^{BW}(q,t) = q^{m+n}H_{m,n}\left(\frac{1}{q},q(t-1)\right).
	\end{displaymath}
\end{corollary}
\begin{proof}
	Using the definition and Proposition~\ref{prop:gamma_bw_ftriangle}, we get
	\begin{align*}
		H_{\Gamma(m,n)}^{BW}(q,t) & = F_{\Gamma(m,n)}^{BW}(q,t-1)\\
		& = q^{m+n}H_{m,n}\left(\frac{1}{q},q(t-1)\right).\qedhere
	\end{align*}
\end{proof}

\begin{proposition}
	For $m,n\geq 0$, the noncrossing matching complex $\Gamma(m,n)$ is homotopy-equivalent to an $n$-dimensional sphere if $m=n$, and to a ball otherwise.
\end{proposition}
\begin{proof}
	By Corollary~\ref{cor:ncm_shellable}, $\Gamma(m,n)$ is shellable, and thus by Theorem~\ref{thm:nonpure_shellable_homotopy} it is homotopy equivalent to a wedge of spheres.  Moreover, this theorem states that the number of $(j-1)$-dimensional spheres involved in this wedge is given by $h_{j,j}$, which is the coefficient of $q^{j}$ in $H_{\Gamma(m,n)}^{BW}(q,t)$.  These coeffients are obtained by evaluating $H_{\Gamma(m,n)}^{BW}(q,0)$.  By Corollary~\ref{cor:gamma_bw_htriangle} and Theorem~\ref{thm:bubble_h_triangle}, we find:
	\begin{align*}
		H_{\Gamma(m,n)}^{BW}(q,0) & = q^{m+n}H_{m,n}\left(\frac{1}{q},-q\right)\\
		& = q^{m+n}\sum_{a\geq 0}\binom{m}{a}\binom{n}{a}q^{-a}\left(\frac{-q}{q}+1\right)^{m+n-2a}\\
		& = q^{m+n}\sum_{a\geq 0}\binom{m}{a}\binom{n}{a}q^{-a}0^{m+n-2a}.
	\end{align*}
	Thus, 
	\begin{displaymath}
		H_{\Gamma(m,n)}^{BW}(q,0) = \begin{cases}0, & \text{if}\;m\neq n,\\x^{n}, & \text{if}\;m=n.\end{cases}
	\end{displaymath}
	The claim now follows.
\end{proof}

\subsection{The Positive Part}
	\label{sec:ncm_positive}
Let us consider the subcomplex $\Gamma^+(m,n)$ of $\Gamma(m,n)$ obtained by deleting all loops.  We may use the $H$-triangle defined in \eqref{eq:bubble_h_triangle} to compute the $f$-polynomial of $\Gamma^{+}(m,n)$.

\begin{proposition}\label{prop:gamma_positive_faces}
    For $m,n\geq 0$,
    \begin{align*}
        f_{\Gamma^{+}(m,n)}(q) = H_{m,n}(q,0) = \sum_{a\geq 0} \binom{m}{a}\binom{n}{a}q^a.
    \end{align*}
\end{proposition}
\begin{proof}
    This follows from \eqref{eq:bubble_h_triangle}, because if we evaluate $0^{0}=1$, then $H_{m,n}(q,0)$ is computed by summing over the shuffle words $\vbf$ for which there exists no $\ubf$ with $\ubf\sindel\vbf$.  In view of the bijection from Proposition~\ref{prop:ncm_faces}, these correspond precisely to the faces of $\Gamma(m,n)$ without loops, hence the faces of $\Gamma^{+}(m,n)$.  The explicit formula follows from Theorem~\ref{thm:bubble_h_triangle}.
\end{proof}

%
%

We will now give another combinatorial explanation of $H_{m,n}(q,0)$.  A \defn{$q$-Delannoy path} is a lattice path from $(0,0)$ to $(m,n)$ using steps of the form $(1,0)$, $(0,1)$ or $q$ differently colored steps of the form $(1,1)$.  We normally use color set $\{1,2,\ldots,q\}$.  Let $\Del_{m,n}(q)$ denote the set of $q$-Delannoy paths from $(0,0)$ to $(m,n)$.  

The $0$-Delannoy paths are the usual northeast paths from $(0,0)$ to $(m,n)$ and the $1$-Delannoy paths are the usual Delannoy paths (see e.g. \cite{banderier:why}).  Figure~\ref{fig:positive_paths} shows the $22$ $2$-Delannoy paths from $(0,0)$ to $(2,2)$.  

\begin{figure}
    \centering
    \includegraphics[page=20,scale=1]{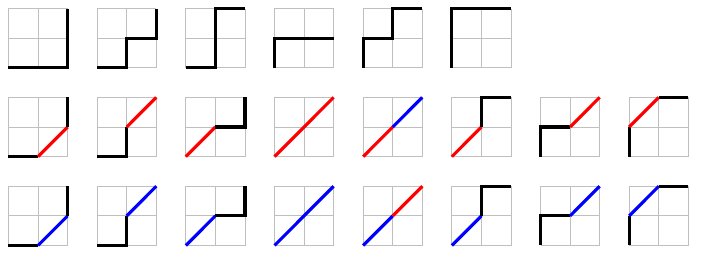}
    \caption{The $22$ $2$-Delannoy paths in the $2\times 2$-square.}
    \label{fig:positive_paths}
\end{figure}

\begin{proposition}
	For $m,n\geq 0$,
	\begin{displaymath}
		\bigl\lvert\Del_{m,n}(q)\bigr\rvert = H_{m,n}(q+1,0).
	\end{displaymath}
\end{proposition}
\begin{proof}
	A $q$-Delannoy path is uniquely determined by the location of its peaks and its colored diagonal steps.  Thus, for every choice of $a$ coordinates $(p_{1},q_{1})$, $(p_{2},q_{2})$, $\ldots$, $(p_{a},q_{a})$ such that $0\leq p_{1}<p_{2}<\cdots<p_{a}<m$ and $1\leq q_{1}<q_{2}<\cdots<q_{a}\leq n$ we may construct $(q+1)^{a}$ different $q$-Delannoy paths as follows: either $(p_{i},q_{i})$ is a peak or it is a diagonal step from $(p_{i},q_{i}-1)$ to $(p_{i}+1,q_{i})$ in one of $q$ colors.  
	
	Since such a sequence of coordinates can be chosen in $\binom{m}{a}\binom{n}{a}$ ways, and $a$ may range from $0$ to $\min\{m,n\}$, we get the stated formula.
\end{proof}

We now describe a bijection between the $q$-Delannoy paths in the $m\times n$-rectangle and \defn{$(q+1)$-flags} in $\Gamma^{+}(m,n)$, \ie tuples $(G_{0},G_{1},\ldots,G_{q})$ of faces of $\Gamma^{+}(m,n)$ such that $G_{0}\subseteq G_{1}\subseteq\cdots\subseteq G_{q}$.  We start with the case $q=0$.  A \defn{peak} of $\mathbf{p}\in\Del_{m,n}(0)$ is a coordinate which is preceded by $N$ and followed by $E$.

\begin{proposition}\label{prop:ncm_faces_paths}
	For $m,n\geq 0$, the faces of dimension $r-1$ of $\Gamma^{+}(m,n)$ are in bijection with the $0$-Delannoy paths in the $(m\times n)$-rectangle with exactly $r$ peaks.
\end{proposition}
\begin{proof}
	Let $G\in\Gamma^{+}(m,n)$ have dimension $r-1$, \ie $G=\bigl\{\{x_{i_{1}},y_{j_{1}}\},\{x_{i_{2}},y_{j_{2}}\},\ldots,\{x_{i_{r}},y_{j_{r}}\}\bigr\}$.  By \eqref{it:nm_1}, we may assume that $1\leq i_{1}<i_{2}<\cdots<i_{r}\leq m$, and by \eqref{it:nm_2} we get $1\leq j_{1}<j_{2}<\cdots<j_{r}\leq n$.  If we mark the coordinates $(i_{1}-1,j_{1}),(i_{2}-1,j_{2}),\ldots,(i_{r}-1,j_{r})$ in an $m\times n$-rectangle, then we can draw a unique $0$-Delannoy path in that rectangle whose $r$ peaks are exactly in the marked locations.  This is certainly an injective procedure.  (If $G=\emptyset$, then this procedure is supposed to produce the path $E^{m}N^{n}$.)
	
	Conversely, let $\mathbf{p}\in\Del_{m,n}(0)$ whose peaks are $\mathbf{p}$ are $(i_{1},j_{1}),(i_{2},j_{2}),\ldots,(i_{r},j_{r})$.  Then, by construction we have $0\leq i_{1}<i_{2}<\cdots<i_{r}<m$ and $1\leq j_{1}<j_{2}<\cdots<j_{r}\leq n$.  It is clear that $\bigl\{\{x_{i_{1}+1},y_{j_{1}}\},\{x_{i_{2}+1},y_{j_{2}}\},\ldots,\{x_{i_{r}+1},y_{j_{r}}\}\bigr\}\in\Gamma^{+}(m,n)$.  Once again, this is an injective assignment which is also the inverse of the construction from the first paragraph.
\end{proof}

\begin{proposition}\label{prop:ncm_flags_paths}
	For $m,n\geq 0$, the $(q+1)$-flags of $\Gamma^{+}(m,n)$ are in bijection with the $q$-Delannoy paths in the $(m\times n)$-rectangle.
\end{proposition}
\begin{proof}
	Let $(G_{0},G_{1},\ldots,G_{q})$ be a $(q+1)$-flag of $\Gamma^{+}(m,n)$.  We start by marking the coordinates in an $m\times n$-rectangle corresponding to the peaks of $G_{0}$ as described in the proof of Proposition~\ref{prop:ncm_faces_paths}.  For $k\in[q]$, we draw a diagonal step in color $k$ from $(s-1,t-1)$ to $(s,t)$ for each $\{x_{s},y_{t}\}\in G_{k}\setminus G_{k-1}$.  The noncrossing condition then implies that we can fill in north and east steps to obtain a unique $q$-Delannoy path whose peaks are those corresponding to $G_{0}$.  
	
	Conversely, if $\mathbf{p}\in\Del_{m,n}(q)$, then its peaks determine a face $G_{0}\in\Gamma^{+}(m,n)$ as described in Proposition~\ref{prop:ncm_faces_paths}.  Suppose that we have already constructed $G_{k-1}$ by considering all diagonal steps of $\mathbf{p}$ of color $<k$.  If $\mathbf{p}$ has a diagonal step of color $k$ from $(s-1,t-1)$ to $(s,t)$, then we add $\{x_{s},y_{t}\}$ to $G_{k-1}$.  By definition, $\mathbf{p}$ cannot have a peak in column $s-1$ or row $t$ nor can it have a diagonal step of color $<k$ connecting columns $s-1$ and $s$ or rows $t-1$ and $t$.  Thus, adding $\{x_{s},y_{t}\}$ to $G_{k}$ does not violate \eqref{it:nm_1} or \eqref{it:nm_2}.  It follows that $G_{k}\in\Gamma^{+}(m,n)$ and $G_{k-1}\subseteq G_{k}$.  Clearly, the resulting $(q+1)$-flag is uniquely determined by $\mathbf{p}$ and this construction is the inverse of the construction from the previous paragraph.
\end{proof}

\begin{figure}
	\centering
	\begin{subfigure}[t]{.45\textwidth}
		\centering
		\includegraphics[page=21,scale=1]{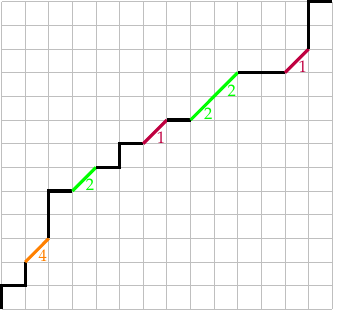}
		\caption{An element of $\Del_{14,13}(4)$.}
		\label{fig:4_path}
	\end{subfigure}
	\hspace*{1cm}
	\begin{subfigure}[t]{.45\textwidth}
		\centering
		\includegraphics[page=22,scale=1]{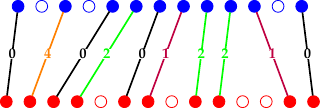}
		\caption{A $5$-flag of $\Gamma^{+}(13,14)$.}
		\label{fig:4_flag}
	\end{subfigure}
	\caption{Illustrating the bijection from $q$-Delannoy paths to $(q+1)$-flags in $\Gamma^{+}$.}
	\label{fig:path_flag_bijection}
\end{figure}

Let us illustrate Proposition~\ref{prop:ncm_flags_paths} by an example.  Let $m=14$, $n=13$ and consider the $4$-Delannoy path $\mathbf{p}$ from Figure~\ref{fig:4_path}.  Assume that the colors are ordered ${\color{purple}1}<{\color{green}2}<{\color{gray}3}<{\color{orange}4}$.  Its peaks are in coordinates $(0,1)$, $(2,5)$, $(5,7)$, $(13,13)$.  We thus get 
\begin{displaymath}
	G_{0} = \bigl\{\{\rd{x_1},\bl{y_1}),\{\rd{x_3},\bl{y_5}),\{\rd{x_6},\bl{y_7}),\{\rd{x_{14}},\bl{y_{13}})\bigr\}.
\end{displaymath}
There are two diagonal steps colored in purple.  These connect the lattice points $(6,7)$ and $(7,8)$ as well as $(12,10)$ and $(13,11)$. Thus,
\begin{displaymath}
	G_{1} = G_{0} \uplus \bigl\{\{\rd{x_7},\bl{y_8}\},\{\rd{x_{13}},\bl{y_{11}}\}\bigr\}.
\end{displaymath}
There are three diagonal steps colored in green: $(3,5)$--$(4,6)$, $(8,8)$--$(9,9)$, $(9,9)$--$(10,10)$.  Thus,
\begin{displaymath}
	G_{2} = G_{1} \uplus \bigl\{\{\rd{x_4},\bl{y_6}\},\{\rd{x_9},\bl{y_9}\},\{\rd{x_{10}},\bl{y_{10}}\}\bigr\}.
\end{displaymath}
There are no diagonal steps colored in gray, thus $G_{3}=G_{2}$.  Finally, there is one diagonal step colored in orange: $(1,2)$--$(2,3)$.  This gives us
\begin{displaymath}
	G_{4} = G_{3} \uplus \bigl\{\{\rd{x_2},\bl{y_3}\}\bigr\}.
\end{displaymath}
Figure~\ref{fig:4_flag} shows the resulting flag $(G_{0},G_{1},G_{2},G_{3},G_{4})$, where the edges are colored according to their introduction.  More precisely, the flag $G_{i}$ consists of all edges labeled $\leq i$.

\begin{remark}
	Notice that when $m=n$, a $q$-Delannoy path $\mathbf{p}\in\Del_{n,n}(q)$ stays weakly below the diagonal if and only if any peak $(i,j)$ of $\mathbf{p}$ satisfies $i\geq j$ and any diagonal step of $\mathbf{p}$ connecting $(s-1,t-1)$ and $(s,t)$ satisfies $s\geq t$.  Let us call such paths \defn{$q$-Schr{\"o}der paths}, because for $q=1$ this definition recovers the usual Schr{\"o}der paths.  For $q=0$, we recover the usual Dyck paths.  
	
	In view of the bijection from Proposition~\ref{prop:ncm_flags_paths}, $0$-Schr{\"o}der paths correspond to \defn{left-leaning faces} of $\Gamma^{+}$, \ie faces where any contained edge $\{x_{s},y_{t}\}$ satisfies $s>t$.  We may thus consider the \defn{left-leaning matching complex} as the subcomplex of $\Gamma^{+}(n,n)$ which consists of all the left-leaning faces.  This complex has Catalan-many faces.  In fact, since the bijection from Proposition~\ref{prop:ncm_faces_paths} maps the number of peaks to the size, the number of faces of the $n$-th left-leaning complex of dimension $k$ is the $(n,k)$-Narayana number $\frac{1}{n}\binom{n}{k}\binom{n}{k+1}$.  Figure~\ref{fig:left_leaning} shows the left-leaning matching complex for $n=3$ and $n=4$.
	
	Under the bijection from Proposition~\ref{prop:ncm_flags_paths}, the $q$-Schr{\"o}der paths would correspond to $(q+1)$-flags in $\Gamma^{+}(n,n)$, which have a left-leaning face as bottom element, and any non-bottom face is \defn{weakly left-leaning}, \ie it may contain edges $\{x_{s},y_{s}\}$.  The $(q+1)$-flags in $\Gamma^{+}(n,n)$ then correspond to \defn{little $q$-Schr{\"o}der paths}, \ie $q$-Schr{\"o}der paths which do not have diagonal steps on the line connecting $(0,0)$ and $(n,n)$.  
\end{remark}

\begin{figure}
	\centering
	\begin{subfigure}[t]{.35\textwidth}
		\centering
		\includegraphics[page=23,scale=.6]{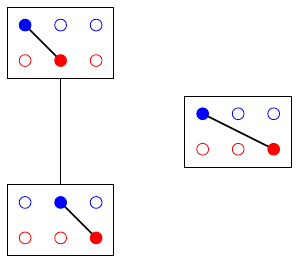}
	\end{subfigure}
	\hspace*{1cm}
	\begin{subfigure}[t]{.55\textwidth}
		\centering
		\includegraphics[page=24,scale=.6]{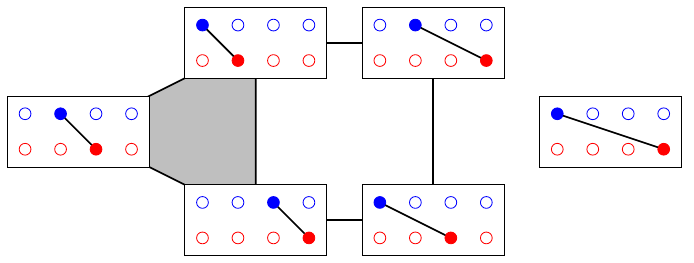}
	\end{subfigure}
	\caption{The left-leaning matching complex for $n=3$ (left) and $n=4$ (right).}
	\label{fig:left_leaning}
\end{figure}

\section{The Noncrossing Bipartite Complex}
	\label{sec:exchange}
In this section, we reinterpret the bubble lattice as a partial ordering on the facets of a simplicial complex. We associate a facet with each shuffle word of $\xbf$ and $\ybf$ so that covering pairs in the bubble lattice correspond to adjacent facets.

\subsection{Definition}

Let $r=m+n$. We introduce two new letters $x_0$ and $y_0$ and define $\wtil{\xbf}\defs x_0x_1\cdots x_m$ and $\wtil{\ybf}\defs y_0y_1\cdots y_n$. Let $\wtil{X}\defs X\uplus\{x_0\}$ and $\wtil{Y}\defs Y\uplus\{y_0\}$ be the supports of $\wtil{\xbf}$ and $\wtil{\ybf}$, respectively. 
Moreover, let $\Ucal\defs\Lcal\uplus\wtil{\Ecal}$, where
\begin{displaymath}
	\Lcal \defs X\uplus Y, \quad\text{and}\quad \wtil{\Ecal} \defs \Ecal(\tilde{X},\tilde{Y})\setminus\bigl\{ \{x_0,y_0\} \bigr\}.
\end{displaymath}
As before, the elements of $\Lcal$ are \defn{loops} and the elements of $\wtil{\Ecal}$ are \defn{edges}.  Recall the definition of $\Ecal(\tilde{X},\tilde{Y})$ from \eqref{eq:edges}.

The \defn{noncrossing bipartite complex} $\Delta\defs\Delta(m,n)$ is the abstract simplicial complex with vertex set $\Ucal$ whose faces $\sigma$ satisfy the following conditions:
\begin{description}
	\item[NB1\label{it:nb_1}] if $\{x,y\}\in\sigma$, then $x\notin\sigma$ and $y\notin\sigma$, \quad and
	\item[NB2\label{it:nb_2}] if $\{x_{s_1},y_{t_1}\}$ and $\{x_{s_2},y_{t_2}\}$ are in $\sigma$ such that $s_1<s_2$ then $t_1<t_2$.
\end{description}
Two elements of $\Ucal$ are \defn{noncrossing} if there exists a face of $\Delta$ containing them both.  They are \defn{crossing} otherwise.  See Figure~\ref{fig:bnc_21} for an illustration of $\Delta(2,1)$.  

\begin{remark}
	We want to emphasize that the noncrossing matching complex $\Gamma(m,n)$ is a subcomplex of the noncrossing bipartite complex.  Indeed, the vertex set $\Tcal$ of $\Gamma(m,n)$ is clearly a subset of the vertex set $\Ucal$ of $\Delta(m,n)$ and Conditions~\eqref{it:nm_2} and \eqref{it:nb_2} are the same.  While Condition~\eqref{it:nm_2} demands that no element of $X\uplus Y$ can appear multiple times in face of $\Gamma(m,n)$, Condition~\eqref{it:nb_1} requires only that elements of $X\uplus Y$ may not appear in a face of $\Delta(m,n)$ as both a loop and as part of an edge.  It is perfectly fine, however, that two edges contained in some face of $\Delta$ have a common element.
\end{remark}

\begin{remark}
	By definition, it is clear that the faces of $\Delta(m,n)$ are bipartite graphs with possible loops, where non-loop edges are only allowed to connect elements from $\wtil{X}$ with elements from $\wtil{Y}$.  Faces of $\Gamma(m,n)$ then correspond in a similar manner to particular matchings of a bipartite graph.  
\end{remark}

%

\begin{figure}
	\begin{subfigure}[t]{.45\textwidth}
		\centering
		\includegraphics[page=13,scale=.8]{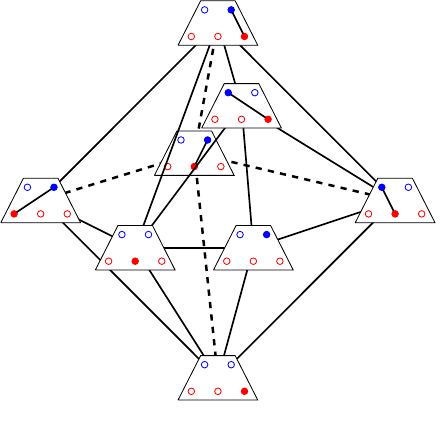}
		\caption{The noncrossing bipartite complex $\Delta(2,1)$.}
		\label{fig:bnc_21}
	\end{subfigure}
	\hspace*{1cm}
	\begin{subfigure}[t]{.45\textwidth}
		\centering
		\includegraphics[page=14,scale=.8]{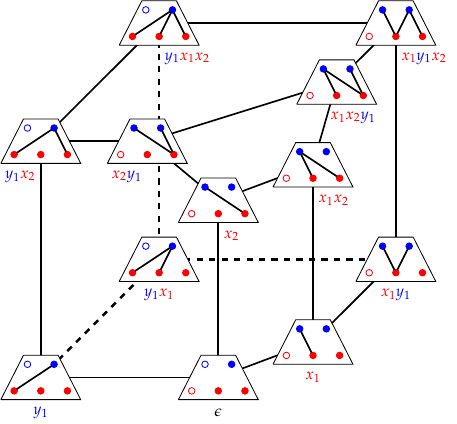}
		\caption{The dual polytope of $\Delta(2,1)$.}
		\label{fig:bnc_21_dual}
	\end{subfigure}
	\caption{The noncrossing bipartite complex and its dual realized as polytopes.}
	\label{fig:bnc_21_main}
\end{figure}

We now establish a bijection between the set of shuffle words of $\xbf$ and $\ybf$ and the set of facets of $\Delta(m,n)$. Under this bijection, the diagram of $\Bub(m,n)$ is an orientation of the dual graph of $\Delta(m,n)$.

Let $\wbf=w_1w_2\cdots w_k$ be a given shuffle word. We let $\wtil{\wbf}=w_{-1}w_0w_1w_2\cdots w_k$ where $w_{-1}=x_0$ and $w_0=y_0$. 
%
We define a map $\phi\colon\Shuf(m,n)\to\wp(\Ucal)$, where $\wp(\Ucal)$ denotes the power set of $\Ucal$, by
\begin{align}
	\phi(\wbf) & \defs \Bigl\{z\in X\uplus Y\colon z\notin\wbf\Bigr\}\\
		& \kern1cm \cup \Bigl\{\{w_{j},w_{i}\}\colon w_{j}\in X,i=\max\{i'<j\colon w_{i'}\in \wtil{Y}\}\Bigr\}\notag\\
		& \kern1cm \cup \Bigl\{\{w_{i},w_{j}\}\colon w_{j}\in Y,i=\max\{i'<j\colon w_{i'}\in \wtil{X}\}\Bigr\}\notag.
\end{align}
For example, if $m=n=3$, then
\begin{equation}\label{eq:phi_bijection}
	\phi(\rd{x_2}\bl{y_2y_3}\rd{x_3}) = \Bigl\{ \rd{x_1},\ \bl{y_1},\ \{\rd{x_2},\bl{y_0}\},\ \{\rd{x_2},\bl{y_2}\},\ \{\rd{x_2},\bl{y_3}\},\ \{\rd{x_3},\bl{y_3}\} \Bigr\}.
\end{equation}

\begin{lemma}\label{lem:phi_noncrossing}
	For $\wbf\in\Shuf(m,n)$, the set $\phi(\wbf)$ is a face of $\Delta(m,n)$ of dimension $r-1$.
\end{lemma}
\begin{proof}
	We first show that any two elements of $\phi(\wbf)$ are noncrossing. We then show that $\phi(\wbf)$ has the appropriate cardinality.

	Edges only exist between letters in the support of $\wtil{\wbf}$, whereas loops only appear at letters not in the support. This means that there are no crossings between loops and edges in $\phi(\wbf)$.  This establishes Condition~\eqref{it:nb_1}.

	Suppose that there are edges $\{x_{s_1},y_{t_1}\}$ and $\{x_{s_2},y_{t_2}\}$ in $\phi(\wbf)$ with $s_1<s_2$ and $t_1>t_2$. Then $x_{s_1}x_{s_2}$ and $y_{t_2}y_{t_1}$ are subwords of $\wtil{\wbf}$.

	If $y_{t_1}$ appears before $x_{s_2}$ in $\wtil{\wbf}$, then the edge $\{x_{s_2},y_{t_2}\}$ is not present in $\phi(\wbf)$ since $y_{t_2}$ is not the last element of $\wtil{Y}$ appearing before $x_{s_2}$ in $\wbf$. Similarly, if $x_{s_2}$ appears before $y_{t_1}$, then the edge $\{x_{s_1},y_{t_1}\}$ is not present in $\phi(\wbf)$ since $x_{s_1}$ is not the last element of $\wtil{X}$ appearing before $y_{t_1}$ in $\wbf$. Either way, we deduce a contradiction.  It follows that Condition~\eqref{it:nb_2} is satisfied.  Hence, $\phi(\wbf)$ is indeed a face of $\Delta(m,n)$.

	It remains to show that $\lvert\phi(\wbf)\rvert=r$. Each element of $X\uplus Y$ not in the support of $\wbf$ corresponds to a loop in $\phi(\wbf)$. Each element of $X\uplus Y$ that is in the support of $\wbf$ is the rightmost letter of exactly one edge in $\phi(\wbf)$. Hence, $\lvert\phi(\wbf)\rvert=\lvert X\uplus Y\rvert = r$.
\end{proof}

\begin{lemma}\label{lem:delta_purity}
	For all $\sigma\in\Delta(m,n)$, there exists a shuffle word $\wbf\in\Shuf(m,n)$ such that $\sigma\subseteq\phi(\wbf)$.
\end{lemma}
\begin{proof}
	Given $\sigma\in\Delta(m,n)$, let $H$ be the graph
	with vertex set $\wtil{X}\uplus\wtil{Y}$ and edge set $\sigma$. To satisfy the noncrossing conditions, each component of $H$ is either an isolated vertex, a vertex with a single loop, or a tree with at least two vertices. We refer to a component of the third type as a tree component.

	Suppose $H'$ is a tree component of $H$. We construct a shuffle word $\ubf_{H'}$ whose support is the set of vertices in $H'$ as follows. If $H'$ only contains a single edge $\{x_s,y_t\}$, then let $\ubf_{H'}=x_sy_t$.

	Otherwise, let $x_s\in X$ and $y_t\in Y$ be the rightmost letters in the vertex set of $H'$. If $\{x_s,y_{t'}\}\in\sigma$ for some $t'<t$, then by the noncrossing condition, there does not exist $\{x_{s'},y_t\}\in\sigma$ for any $s'<s$. The same holds true when swapping the roles of $x_s$ and $y_t$. Since $H'$ is connected, $H'$ contains the edge $\{x_s,y_t\}$, and one of the vertices $x_s,\ y_t$ is a leaf. Let $H''$ be the subtree obtained by deleting this leaf. If $z\in\{x_s,y_t\}$ is the leaf, then we set $\ubf_{H'} = \ubf_{H''}z$.

	We form a word $\wtil{\wbf}$ by concatenating each of the words $\ubf_{H'}$ for each of the tree components $H'$ of $H$. By the noncrossing condition, there exists (a unique) way to concatenate these words so that $\wtil{\wbf}$ is a shuffle word of $\wtil{\xbf}$ and $\wtil{\ybf}$. Let $\wbf$ be the subword of $\wtil{\wbf}$ without $x_0$ or $y_0$.

	We claim that $\sigma\subseteq\phi(\wbf)$. It is straightforward to check that $\phi(\wbf)$ contains all of the edges in tree components of $H$ by induction. Any $z\in X\uplus Y$ not in any tree component of $H$ is not part of the support of $\wbf$, which means $z\in\phi(\wbf)$. Hence, the loop components are covered as well.
\end{proof}

\begin{example}
	Let $\sigma$ be the following face of $\Delta(3,3)$.  
	\begin{center}
		\includegraphics[page=15,scale=1]{shuffle_figures.pdf}
	\end{center}
	We label the red nodes (bottom) from left to right by $x_{0},x_{1},x_{2},x_{3}$ and the blue ones (top) by $y_{0},y_{1},y_{2},y_{3}$.  The solid nodes are contained in $\sigma$, the hollow ones are not.  Solid nodes without attached lines are loop components, solid nodes with attached lines belong to edge components.  Thus, 
	\begin{displaymath}
		\sigma = \Bigl\{ \rd{x_1},\ \{\rd{x_2},\bl{y_0}\},\ \{\rd{x_2},\bl{y_2}\},\ \{\rd{x_3},\bl{y_3}\} \Bigr\}.
	\end{displaymath}
	
	There are two isolated vertices $\rd{x_0}$, $\bl{y_1}$, a single loop component $\rd{x_1}$ and two tree components $H'_{1}=\bigl\{\{\rd{x_2},\bl{y_0}\},\{\rd{x_2},\bl{y_2}\}\bigr\}$ and $H'_{2}=\bigl\{\{\rd{x_3},\bl{y_3}\}\bigr\}$.  The associated words are $\ubf_{H'_{1}}=\rd{x_2}\bl{y_0y_2}$ and $\ubf_{H'_{2}}=\rd{x_3}\bl{y_3}$ so that we obtain $\wtil{\wbf}=\rd{x_2}\bl{y_0y_2}\rd{x_3}\bl{y_3}$.  Finally, we obtain the shuffle word $\wbf=\rd{x_2}\bl{y_2}\rd{x_3}\bl{y_3}$, and we have
	\begin{displaymath}
		\phi(\wbf) = \Bigl\{\rd{x_1},\ \bl{y_1},\ \{\rd{x_2},\bl{y_0}\},\ \{\rd{x_2},\bl{y_2}\},\ \{\rd{x_3},\bl{y_2}\},\ \{\rd{x_3},\bl{y_3}\}\Bigr\}\supseteq\sigma.
	\end{displaymath}	
	By \eqref{eq:phi_bijection}, we observe that $\sigma\subseteq\phi(\rd{x_2}\bl{y_2y_3}\rd{x_3})$ as well, and we have $\rd{x_2}\bl{y_2}\rd{x_3}\bl{y_3}\bubcov\rd{x_2}\bl{y_2y_3}\rd{x_3}$.
\end{example}

\begin{lemma}\label{lem:delta_thin}
  Let $\sigma\in\Delta(m,n)$ such that $\dim\sigma = r-2$. There exist exactly two distinct facets containing $\sigma$.
\end{lemma}
\begin{proof}
  Let $H$ be the graph with vertex set $\wtil{X}\uplus\wtil{Y}$ and edge set $\sigma$. The components of $H$ are either isolated vertices, loops, or trees with at least one edge. Since the dimension of $\sigma$ is one less than that of any facet containing it, one of the following must be true of the graph $H$.
  \begin{enumerate}
  \item There is no tree component and $m+n-1$ loops.
  \item There is one isolated vertex among $X\uplus Y$ and exactly one tree component.
  \item There are no isolated vertices and exactly two tree components.
  \end{enumerate}
  Any other possibility would allow too many additional noncrossing edges to be added to $\sigma$. We examine each case and witness exactly two ways to extend $\sigma$ to a maximal noncrossing collection.

  In the first case, we may either add a loop at the remaining element of $X\uplus Y$ or add an edge to $x_0$ or $y_0$ as appropriate. In the second case, we can either add a loop to the isolated vertex or attach the isolated vertex to the tree component in exactly one way. In the third case, an edge may be added between the vertices of the rightmost edge of the left tree component and those of the leftmost edge of the right tree component in two ways.
\end{proof}

\begin{figure}
	\centering
	\includegraphics[page=16,scale=1]{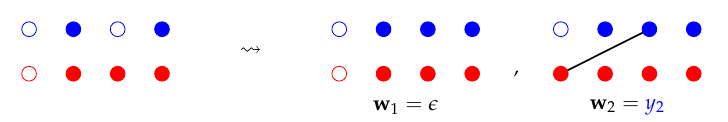}\\
	\includegraphics[page=17,scale=1]{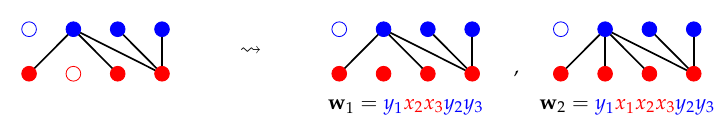}\\
	\includegraphics[page=18,scale=1]{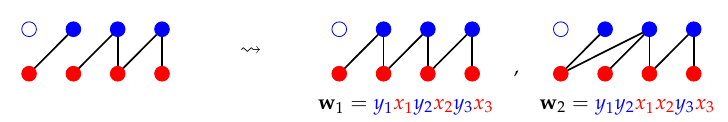}
	\caption{The three types of extensions of ridges in the noncrossing bipartite complex.}
	\label{fig:ridge_extensions}
\end{figure}

The three types of extensions mentioned in the proof of Lemma~\ref{lem:delta_thin} are illustrated in Figure~\ref{fig:ridge_extensions}.  Lemmas~\ref{lem:phi_noncrossing}, ~\ref{lem:delta_purity}, and~\ref{lem:delta_thin} imply that $\Delta(m,n)$ is a pure, thin simplicial complex of dimension $r-1$. The following lemmas show that the Hasse diagram of the bubble lattice is an orientation of the dual graph of the complex $\Delta(m,n)$.

\begin{lemma}
  The function $\phi$ from $\Shuf(m,n)$ to facets of $\Delta(m,n)$ is a bijection.
\end{lemma}
\begin{proof}
	Lemma~\ref{lem:delta_purity} implies that $\phi$ is a surjective function. It remains to show that distinct shuffle words map to distinct facets of $\Delta(m,n)$.

	Let $\ubf=u_1u_2\cdots u_k$ and $\vbf=v_1v_2\cdots v_{\ell}$ be distinct shuffle words of $\xbf$ and $\ybf$. If $\ubf$ and $\vbf$ have different supports, then the set of loops in $\phi(\ubf)$ and $\phi(\vbf)$ is different. So we assume that $\ubf$ and $\vbf$ have the same support.

	Suppose $j$ is the smallest index such that $u_j\neq v_j$. Since $\ubf$ and $\vbf$ are shuffle words of $\xbf$ and $\ybf$ with the same support, either $u_j\in X$ and $v_j\in Y$ or vice versa. Without loss of generality, we assume $u_j\in X$ and $v_j\in Y$. There is a unique $i<j$ such that $\{u_j,u_i\}$ is in $\phi(\ubf)$. But, since $v_j\in Y$ appears between $u_i\in Y$ and $u_j\in X$ in $\vbf$, the edge $\{u_j,u_i\}$ is not in $\phi(\vbf)$. 
\end{proof}

Given a shuffle word $\wbf\in\Shuf(m,n)$, we let $F_{\wbf}\defs\phi(\wbf)$ be the facet of $\Delta(m,n)$ corresponding to $\wbf$.  Moreover, for $\sigma\in\Delta(m,n)$, we define
\begin{displaymath}
	K(\sigma) \defs \bigl\{\wbf\in\Shuf(m,n)\colon \sigma\subseteq F_{\wbf}\bigr\}.
\end{displaymath}

\begin{lemma}\label{lem:facial_interval}
	 For any $\sigma\in\Delta(m,n)$, the set $K(\sigma)$
	 is a closed interval of $\Bub(m,n)$.
\end{lemma}
\begin{proof}
    Before proving the result for a general face $\sigma$, we prove it for vertices of $\Delta(m,n)$.  
    If $\sigma$ is a vertex of $\Delta(m,n)$, there are two options.

    (i) First, we consider $\sigma=\{e\}$ with $e=\{x_{s},y_{t}\}\in\wtil{\Ecal}$.  Then, any $\wbf\in K(\sigma)$ must contain an adjacent subsequence $x_{s}y_{t_{1}}\cdots y_{t_{k}}y_{t}$ for $t_{1}<\cdots<t_{k}<t$ or $y_{t}x_{s_{1}}\cdots x_{s_{\ell}}x_{s}$ for $s_{1}<\cdots<s_{\ell}<s$.  This implies that the minimal element in $K(\sigma)$ (among all shuffle words in $\Shuf(m,n)$) is $\wbf_{0}=x_{1}x_{2}\cdots x_{s}y_{t}$ and the maximal element in $K(\sigma)$ is $\wbf_{1}=y_{1}y_{2}\cdots y_{t}x_{s}$.  Lemma~\ref{lem:bubble_order} then implies that $\wbf_{0}\bubleq\wbf_{1}$.  Let $\wbf\in\Shuf(m,n)$ with $\wbf_{0}\bubleq\wbf\bubleq\wbf_{1}$.  If we can show that $\wbf\in K(\sigma)$, then we are done.  If $x_{s}$ and $y_{t}$ are adjacent in $\wbf$, then clearly $\wbf\in K(\sigma)$.  Otherwise, we distinguish two cases.\\
    (ia) Suppose that $(x_{s},y_{t})\in\invset(\wbf)$.  Then $e\notin F_{\wbf}$ if and only if there is some $y_{t'}$ which is strictly between $y_{t}$ and $x_{s}$.  By design, we must have $t<t'$.  However, $\wbf_{\ybf}\subseteq\{y_{1},y_{2},\ldots,y_{t}\}$.  Thus, $e\in F_{\wbf}$.\\
    (ib) Suppose that $(x_{s},y_{t})\notin\invset(\wbf)$.  Then $e\notin F_{\wbf}$ if and only if there is some $x_{s'}$ which is strictly between $x_{s}$ and $y_{t}$.  By design, we must have $s<s'$.  However, $\wbf_{\xbf}\subseteq\{x_{1},x_{2},\ldots,x_{s}\}$.  Thus, $e\in F_{\wbf}$.\\
    It follows that $\wbf\in K(\sigma)$ and therefore $K(\sigma)=[\wbf_{0},\wbf_{1}]$.

    (ii) Now, consider $\sigma=\bigl\{z\bigr\}$ with $z\in\Lcal$.  Then, no $\wbf\in K(\sigma)$ contains the letter $z$.  Thus, if $z=x_{s}$, then $K(\sigma)=[\xbf_{\hat{s}},\ybf]$ and if $z=y_{t}$, then $K(\sigma)=[\xbf,\ybf_{\hat{t}}]$.

    \medskip

    Now suppose that $\sigma\in\Delta(m,n)$ is an arbitrary face.  We claim that $K(\sigma)=\bigcap_{e\in\sigma}K\bigl(\{e\}\bigr)$.  Indeed, let $\wbf\in K(\sigma)$.  Then, for every $e\in\sigma$, we have $\{e\}\subseteq\sigma\subseteq F_{\wbf}$ and therefore $\wbf\in K\bigl(\{e\}\bigr)$.  It follows that $\wbf\in\bigcap_{e\in\sigma}K\bigl(\{e\}\bigr)$.

    Conversely, let $\wbf\in\bigcap_{e\in\sigma}K\bigl(\{e\}\bigr)$.  This means that $\wbf\in K\bigl(\{e\}\bigr)$ for all $e\in\sigma$.  Hence, $e\in F_{\wbf}$, which implies immediately that $\sigma\subseteq F_{\wbf}$ and thus $\wbf\in K(\sigma)$.

    We have already argued that $K\bigl(\{e\}\bigr)$ constitutes a closed interval of $\Bub(m,n)$ when $e\in\Ucal$ is a vertex of $\Delta(m,n)$. It is an easy fact that intersections of closed intervals in finite lattices are closed intervals again.  Thus, the previous reasoning implies that $K(\sigma)$ is a closed interval of $\Bub(m,n)$ for all $\sigma\in\Delta(m,n)$.
\end{proof}

In order to illustrate Lemma~\ref{lem:facial_interval}, we pick $m=n=3$.  Then, $\sigma=\bigl\{\bl{y_{2}},\{\rd{x_{2}},\bl{y_{1}}\},\{\rd{x_{2}},\bl{y_{3}}\}\bigr\}$ is a face of $\Delta(3,3)$, and the induced interval $K(\sigma)$ in $\Bub(3,3)$ is shown in Figure~\ref{fig:bubble_33_interval}.  
The word constructed from $\sigma$ in the proof of Lemma~\ref{lem:delta_purity} is $\rd{x_{2}}\bl{y_{1}y_{3}}$ which is neither the minimal nor the maximal element of $K(\sigma)$.

\begin{figure}
    \centering
    \includegraphics[scale=1,page=19]{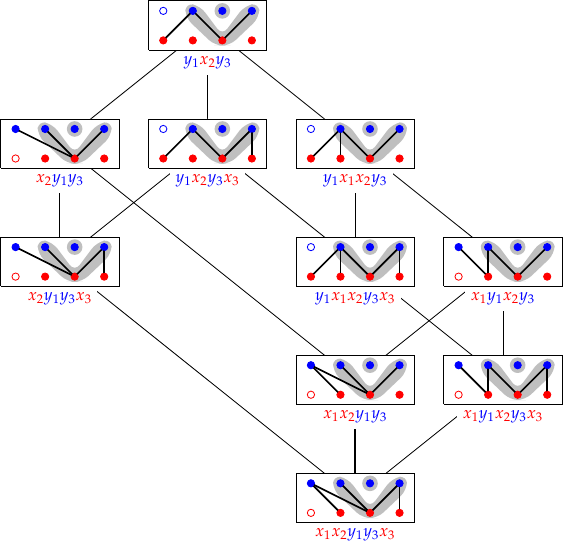}
    \caption{The interval in $\Bub(3,3)$ induced by the (highlighted) face $\sigma=\bigl\{\bl{y_{2}},\{\rd{x_{2}},\bl{y_{1}}\},\{\rd{x_{2}},\bl{y_{3}}\}\bigr\}$.}
    \label{fig:bubble_33_interval}
\end{figure}

\begin{lemma}\label{lem:facial_covers}
  For shuffle words $\ubf,\vbf$, we have $\dim F_{\ubf}\cap F_{\vbf} = r-2$ if and only if either $\ubf\bubcov\vbf$ or $\vbf\bubcov \ubf$.
\end{lemma}
\begin{proof}
	Suppose $\ubf\bubcov\vbf$ is a covering pair in the bubble lattice. Then either $\ubf\sindel\vbf$ or $\ubf\transpose\vbf$. If $\vbf$ is obtained from $\ubf$ by deleting $x_s$, then $F_{\vbf}$ is obtained from $F_{\ubf}$ by adding a loop at $x_s$ and deleting an edge $\{x_s,y_t\}$ for some $y_t$ to the left of $x_s$ in $\ubf$. A similar result holds if $\vbf$ is obtained from $\ubf$ by inserting $y_t$. If $\vbf$ is obtained from $\ubf$ by transposing an adjacent pair $x_sy_t$, then $F_{\vbf}$ is obtained from $F_{\ubf}$ by deleting an edge $\{x_s,y_{t^{\pr}}\}$ and adding an edge $\{x_{s^{\pr}},y_t\}$ for some $t^{\pr}<t$ and $s^{\pr}<s$. In each of these cases, the adjacent shuffle words correspond to adjacent facets of $\Delta(m,n)$.

	Now assume $F_{\ubf}\cap F_{\vbf}$ has dimension $r-2$. By Lemma~\ref{lem:delta_thin}, $K(F_{\ubf}\cap F_{\vbf})=\{\ubf,\vbf\}$, and by Lemma~\ref{lem:facial_interval} this is a closed interval in $\Bub(m,n)$.  Thus, we have $\ubf\bubcov\vbf$ or $\vbf\bubcov\ubf$.
%
%
\end{proof}

We conclude this section by proving that every linear extension of $\Bub(m,n)$ induces a shelling order of the facets of $\Delta(m,n)$.

\begin{proposition}\label{prop:bub_extension_shelling}
	If $\wbf_1,\wbf_2,\ldots$ is a linear extension of $\Bub(m,n)$, then $F_{\wbf_1},F_{\wbf_2},\ldots$ is a shelling order of $\Delta(m,n)$.
\end{proposition}
\begin{proof}
  Fix a linear extension $\wbf_1,\wbf_2,\ldots$ of the bubble lattice $\Bub(m,n)$. Let $i<j$ be given, and let $\sigma=F_{\wbf_i}\cap F_{\wbf_j}$. To prove that $F_{\wbf_1},F_{\wbf_2},\ldots$ is a shelling order, we find a word $\wbf_k$ with $k<j$ such that $\sigma\subseteq F_{\wbf_k}\cap F_{\wbf_j}$ and $\dim F_{\wbf_k}\cap F_{\wbf_j}=r-2$.

  By Lemma~\ref{lem:facial_interval}, the set of facets containing $\sigma$ corresponds to an interval $I$ of $\Bub(m,n)$. The words $\wbf_i$ and $\wbf_j$ are both in $I$. Since $\wbf_i$ precedes $\wbf_j$ in a linear extension of $\Bub(m,n)$, we know that $\wbf_j$ is not the minimum element of $I$. Hence, there exists a lower cover $\wbf_k\bubcov \wbf_j$ such that $\wbf_k\in I$. Consequently, $\sigma\subseteq F_{\wbf_k}\cap F_{\wbf_j}$, and by Lemma~\ref{lem:facial_covers}, we have $\dim F_{\wbf_k}\cap F_{\wbf_j}=r-2$. Finally, we know $k<j$ since $\wbf_k$ must appear before $\wbf_j$ in any linear extension.
\end{proof}

\subsection{Face Enumeration in the Noncrossing Bipartite Complex}
	\label{sec:delta_enumeration}
In this section, we study the face enumeration in $\Delta(m,n)$ and relate it to the face enumeration of $\Gamma(m,n)$.  

By Lemma~\ref{lem:ncm_face_polynomial}, the $H$-triangle of $\Bub(m,n)$ is a refined face-enumerating polynomial of $\Gamma(m,n)$, where refinement is with respect to the two types of vertices of $\Gamma(m,n)$.  By definition, these can either be loops or edges.  In Remark~\ref{rem:fh_correspondence} we have predicted the existence of another related polynomial that can be obtained from the $H$-triangle by an invertible substitution of variables.  In this section we want to give a combinatorial realization of this polynomial as a certain face-enumerating polynomial of $\Delta(m,n)$.  

Let us make this precise, and define the \defn{$F$-triangle} of $\Delta(m,n)$ by
\begin{equation}
	F_{m,n}(q,t) \defs \sum_{\sigma\in\Delta} q^{|\sigma \cap \wtil{\Ecal}|} t^{|\sigma \cap \Lcal|}.
\end{equation}
It is immediately clear that $F_{m,n}(q,q)=f_{\Delta(m,n)}(q)$.  We prove the following formula for the $F$-triangle.  

\begin{theorem}\label{thm:Ftriangle}
	For $m,n\geq 0$, we have
	\begin{displaymath}
		F_{m,n}(q,t) = \sum_{a=0}^{\min\{m,n\}} \binom{m}{a}\binom{n}{a} q^a(q+1)^a(q+t+1)^{m+n-2a}.
	\end{displaymath}
\end{theorem}

Rather than proving Theorem~\ref{thm:Ftriangle} by direct enumeration, we instead prove the identity
\begin{equation}\label{eq:FH_identity}
	H_{m,n}(q,t) = (q-1)^{m+n} F_{m,n}\left(\frac{1}{q-1}, \frac{1+q(t-1)}{q-1}\right),
\end{equation}
which is equivalent to the one stated in our third main theorem: Theorem~\ref{thm:FH_identity}, and the one mentioned in Remark~\ref{rem:fh_correspondence}.  We then deduce the formula for the $F$-triangle from Theorem~\ref{thm:bubble_h_triangle}. Equation~\eqref{eq:FH_identity} will be proved using a further refinement of the $F$-triangle and $H$-triangle.  The \defn{(extended) $F$-triangle} of $\Delta(m,n)$ is the polynomial in $m+n+1$ variables:
\begin{equation}\label{eq:extended_f_triangle}
	\wtil{F}_{m,n}(q;t_{x_1},\ldots,t_{y_n}) \defs \sum_{\sigma\in\Delta}q^{\lvert\sigma \cap \wtil{\Ecal}\rvert} \prod_{z\in \sigma} t_z.
\end{equation}
The \defn{(extended) $H$-triangle} of $\Bub(m,n)$ is:
\begin{equation}\label{eq:extended_h_triangle}
	\wtil{H}_{m,n}(q;t_{x_{1}},\ldots,t_{y_{n}}) \defs \sum_{\ubf\in\Shuf(m,n)}q^{\indeg(\ubf)}\prod_{z\in X\uplus Y}t_{z}^{\lvert\{\ubf'\lessdot\ubf\colon\lambda(\ubf',\ubf)=z\}\rvert}.
\end{equation}
Clearly, we have $\wtil{H}_{m,n}(q;t,t,\ldots,t)=H_{m,n}(q,t)$.  

\begin{theorem}\label{thm:extendedFH}
	Let $\wtil{F}$ and $\wtil{H}$ be the extended $F$-triangle and $H$-triangle as defined in \eqref{eq:extended_f_triangle} and \eqref{eq:extended_h_triangle}. Then
	\begin{displaymath}
		\wtil{H}(q; t_{x_1},\ldots, t_{y_n}) = (q-1)^r \wtil{F}\left(\frac{1}{q-1}; \frac{1+ q(t_{x_1}-1)}{q-1}, \ldots, \frac{1+ q(t_{y_n}-1)}{t-1}\right).
	\end{displaymath}
\end{theorem}

\begin{example}
	By inspection of Figure~\ref{fig:bnc_21}, we obtain
	\begin{multline*}
		\wtil{F}_{2,1}(q;t_{x_{1}},t_{x_{2}},t_{y_{1}}) = 3q^{3} + 2q^{2}t_{x_{1}} + 2q^{2}t_{x_{2}} + q^{2}t_{y_{1}} + qt_{x_{1}}t_{x_{2}} + qt_{x_{1}}t_{y_{1}} + qt_{x_{2}}t_{y_{1}} + t_{x_{1}}t_{x_{2}}t_{y_{1}}\\
			 + 7q^{2} + 3qt_{x_{1}} + 3qt_{x_{2}} + 2qt_{y_{1}} + t_{x_{1}}t_{x_{2}} + t_{x_{1}}t_{y_{1}} + t_{x_{2}}t_{y_{1}} + 5q + t_{x_{1}} + t_{x_{2}} + t_{y_{1}} + 1,
	\end{multline*}
	which specializes to
	\begin{displaymath}
		F_{2,1}(q,t) = 3q^{3} + 5q^{2}t + 3qt^{2} + t^{3} + 7q^{2} + 8qt + 3t^{2} + 5q + 3t + 1.
	\end{displaymath}
	Likewise, by inspection of Figure~\ref{fig:bubble_21_labeled}, we obtain
	\begin{multline*}
		\wtil{H}_{2,1}(q;t_{x_{1}},t_{x_{2}},t_{y_{1}}) = q^{3}t_{x_{1}}t_{x_{2}}t_{y_{1}} + q^{2}t_{x_{1}}t_{x_{2}} + q^{2}t_{x_{1}}t_{y_{1}} + q^{2}t_{x_{2}}t_{y_{1}}\\
			+ q^{2}t_{x_{1}} + q^{2}t_{x_{2}} + qt_{x_{1}} + qt_{x_{2}} + qt_{y_{1}} + 2q + 1,
	\end{multline*}
	which specializes to
	\begin{displaymath}
		H_{2,1}(q,t) = q^{3}t^{3} + 3q^{2}t^{2} + 2q^{2}t + 3qt + 2q + 1.
	\end{displaymath}
	We can check quickly that the relation from Theorem~\ref{thm:extendedFH} is satisfied.
\end{example}

As a first step towards a proof of Theorem~\ref{thm:extendedFH}, we establish a simpler identity that is key to the proof.

\begin{lemma}\label{lem:fh_identity}
	Let $\Delta=\Delta(m,n)$ and $\Gamma=\Gamma(m,n)$ be the noncrossing bipartite complex and noncrossing matching complex, respectively. Then $h_{\Delta}(q) = f_{\Gamma}(q)$.
\end{lemma}
\begin{proof}
	Let $\wbf_1,\wbf_2,\ldots,\wbf_N$ be a linear extension of $\Bub(m,n)$. Then, $F_{\wbf_1},F_{\wbf_2},\ldots,F_{\wbf_N}$ is a shelling order of $\Delta(m,n)$ by Proposition~\ref{prop:bub_extension_shelling}. By Proposition~\ref{prop:h_vector_restriction}, the $h$-polynomial of $\Delta$ is
	\begin{displaymath} 
		h_{\Delta}(q) = \sum_{j=1}^N t^{\lvert\Res(F_{\wbf_j})\rvert}, 
	\end{displaymath}
	where $\Res(F_{\wbf_j})$ is the minimal face of $F_{\wbf_j}$ that does not yet appear in the subcomplex generated by $\{F_{\wbf_i} \colon i<j\}$. 
	
	By Lemma~\ref{lem:facial_covers}, the size of $\Res(F_{\wbf})$ is equal to the number of upper covers of $\wbf$, which in view of Lemma~\ref{lem:bubble_hasse_regular} equals $m+n-\indeg(\wbf)$. It follows that
	\begin{displaymath}
		h_{\Delta}(q) = \sum_{\wbf\in\Shuf(m,n)}q^{m+n-\indeg(\wbf)} = q^{m+n}H_{m,n}\left(\frac{1}{q},1\right) \overset{\tiny\text{Lem.~}\ref{lem:ncm_face_polynomial}}{=} q^{m+n}f_{\Gamma}\left(\frac{1}{q}\right).
	\end{displaymath}
	Now, Lemma~\ref{lem:bubble_duality} states that $\Bub(m,n)$ is isomorphic to the dual lattice of $\Bub(n,m)$ which implies that the number of 
	shuffle words with $k$ lower covers equals the number of shuffle words with $k$ upper covers.  It follows that
	\begin{equation}\label{eq:gamma_f_symmetry}
		q^{m+n}f_{\Gamma}\left(\frac{1}{q}\right) = f_{\Gamma}(q)
	\end{equation}
	which concludes the proof.
\end{proof}

\begin{corollary}\label{cor:delta_dehn_sommerville}
	The noncrossing bipartite complex $\Delta(m,n)$ satisfies the Dehn--Sommerville relations, \ie we have $h_{i}=h_{m+n-i}$ for all $0\leq i\leq m+n$.  
\end{corollary}

\begin{remark}
	Corollary~\ref{cor:delta_dehn_sommerville} would follow immediately if we knew that $\Delta(m,n)$ was the boundary of a polytope (\cite{ziegler:polytopes}*{Theorem~8.21}).  We suspect that this is the case, but we currently do not have a proof of this property.
\end{remark}

Given a subset $I\subseteq X\uplus Y$, we define the subcomplex $\Delta_I$ to be the full subcomplex on the ground set $\Ucal_I \defs \Lcal_I\uplus \wtil{\Ecal}_I$ where
\begin{displaymath}
	\Lcal_I \defs \bigl\{ z\in\Lcal \colon z\notin I\bigr\}, \quad\text{and}\quad \wtil{\Ecal}_I \defs \bigl\{ \{x,y\} \in \wtil{\Ecal} \colon x,y\notin I \bigr\}.
\end{displaymath}
In other words, the subcomplex $\Delta_I$ is obtained by deleting all loops and edges containing an element in $I$. That is, $\Delta_{I}=\lnk_{\Delta}(I)$.  We similarly define $\Gamma_I$ to be the link of $I$ in $\Gamma(m,n)$.  Clearly, $\Delta_I$ (resp. $\Gamma_{I}$) is isomorphic to $\Delta(m', n')$ (resp. $\Gamma(m',n')$), where $m' = m - \lvert X\cap I\rvert$ and $n' = n - \lvert Y\cap I\rvert$.

\begin{proof}[Proof of Theorem~\ref{thm:extendedFH}]
	Using the definition of the extended $F$-triangle and expanding the product, we have
	\begin{align*}
		(q-1)^r \wtil{F}_{m,n}\bigg(\frac{1}{q-1}; &\frac{1+ q(t_{x_1}-1)}{q-1}, \ldots, \frac{1+ q(t_{y_n}-1)}{q-1}\bigg)\\
		& = \sum_{\sigma\in\Delta} (q-1)^{r-|\sigma|} \prod_{z\in\sigma}\bigl(1+q(t_z-1)\bigr)\\
		& = \sum_{\sigma\in\Delta}(q-1)^{r-|\sigma|}\sum_{I\subseteq\Lcal\cap \sigma} q^{|I|}\prod_{z\in I} (t_z-1)\\
		& = \sum_{\sigma\in\Delta}(q-1)^{r-|\sigma|}\sum_{I\subseteq\Lcal\cap \sigma} q^{|I|}\sum_{J\subseteq I}(-1)^{|I|-|J|}\prod_{z\in J}t_z.
	\end{align*}
	Rearranging the sums gives the equivalent expression
	\begin{equation}\label{eq:A}
	  \sum_{J\subseteq\Lcal} \prod_{z\in J} t_z \sum_{J\subseteq I\subseteq \Lcal} (-1)^{|I|-|J|}q^{|I|}\sum_{\sigma\in\Delta_I}(q-1)^{r-|\sigma|-|I|}.
	\end{equation}
	The inner sum is equal to 
	\begin{displaymath}
		(q-1)^{r-\lvert I\rvert}f_{\Delta_{I}}\left(\frac{1}{q-1}\right) \overset{\eqref{eq:fh_relation}}{=} q^{r-\lvert I\rvert}h_{\Delta_{I}}\left(\frac{1}{q}\right) \overset{\tiny\text{Lem.~}\ref{lem:fh_identity}}{=} q^{r-\lvert I\rvert}f_{\Gamma_{I}}\left(\frac{1}{q}\right).
	\end{displaymath}
	As we have remarked in the paragraph before this proof, $\Gamma_{I}$ is isomorphic to $\Gamma(m',n')$, where $m'=m-\lvert X\cap I\rvert$ and $n'=n-\lvert Y\cap I\rvert$, so that we may use \eqref{eq:gamma_f_symmetry} to conclude that
	\begin{displaymath}
		\sum_{\sigma\in\Delta_I}(q-1)^{r-|\sigma|-|I|} = f_{\Gamma_{I}}(q).
	\end{displaymath}
	Hence, \eqref{eq:A} is equal to
	\begin{displaymath}
		\sum_{J\subseteq\Lcal} \prod_{z\in J} t_z \sum_{J\subseteq I\subseteq\Lcal}(-1)^{|I|-|J|}q^{|I|}f_{\Gamma_I}(q).
	\end{displaymath}
  	Using the Principle of Inclusion-Exclusion, the inner sum simplifies to $f_{\Gamma_J^+}(q)$, which gives us
	\begin{equation*}
		(q-1)^r \wtil{F}_{m,n}\left(\frac{1}{q-1}; \frac{1+ q(t_{x_1}-1)}{q-1}, \ldots, \frac{1+ q(t_{y_n}-1)}{q-1}\right) = \sum_{J\subseteq\Lcal} \prod_{z\in J} t_z f_{\Gamma_{J}^{+}}(q).
	\end{equation*}
	Finally, we expand the extended $H$-triangle:
	\begin{align*}
    	\wtil{H}_{m,n}(q; t_{x_1},\ldots, t_{y_n}) &= \sum_{\sigma\in \Gamma} q^{|\sigma\cap \wtil{\Ecal}|} \prod_{z\in \sigma}t_z\\
	    & = \sum_{J\subseteq\Lcal} \prod_{z\in J} t_z f_{\Gamma_J^+}(q).\qedhere
	\end{align*}
\end{proof}

\begin{proof}[Proof of Theorem~\ref{thm:FH_identity}]
	This follows from Theorem~\ref{thm:extendedFH} by setting $t_{x_{1}}=\cdots=t_{y_{n}}$.
\end{proof}

\begin{remark}
    For $n=1$, we obtain
    \begin{align*}
        F_{m,1}(q,t) & = \sum_{a=0}^{1} \binom{m}{a} q^a(q+1)^a(q+t+1)^{m+1-2a}\\
        & = (q+t+1)^{m-1}\Bigl((m+1)q^2 + 2qt + (m+2)q + (t+1)^{2}\Bigr).
    \end{align*}
    This agrees with the $F$-triangle for the Hochschild lattice $\Hoch(m+1)$ computed in \cite{muehle:hochschild}*{Corollary~6.7}, which comes to no surprise in view of \cite{mcconville.muehle:bubbleI}*{Proposition~4.19}, stating that $\Bub(m,1)$ is isomorphic to $\Hoch(m+1)$.
\end{remark}

The $F$-triangle has a certain symmetry property which according to Chapoton ``is a refined version of the classical Dehn--Sommerville equations for complete simplicial fans''; \cite{chapoton:enumerative}*{Page 4}.

\begin{corollary}
	For $m,n\geq 0$, we have
	\begin{displaymath}
		F_{m,n}(q,t) = (-1)^{m+n} F_{m,n}(-1-q,-1-t).
	\end{displaymath}
\end{corollary}
\begin{proof}
	By Theorem~\ref{thm:Ftriangle}, we get
	\begin{align*}
		(-1)^{m+n}F_{m,n} & (-1-q,-1-t)\\
			& = (-1)^{m+n}\sum_{a=0}^{\min\{m,n\}} \binom{m}{a}\binom{n}{a} (-1-q)^a(-q)^a(-1-q-t)^{m+n-2a}\\
			& = \sum_{a=0}^{\min\{m,n\}} \binom{m}{a}\binom{n}{a} q^a(q+1)^a(q+t+1)^{m+n-2a}\\
			& = F_{m,n}(q,t).\qedhere
	\end{align*}
\end{proof}

We may now conclude that $\Delta(m,n)$ is homotopy equivalent to a sphere.

\begin{proposition}\label{prop:delta_sphere}
	For $m,n\geq 0$, the noncrossing bipartite complex $\Delta(m,n)$ is homotopy equivalent to an $(m+n-1)$-dimensional sphere.
\end{proposition}
\begin{proof}
	By Lemmas~\ref{lem:phi_noncrossing} and \ref{lem:delta_purity}, $\Delta(m,n)$ is a pure simplicial complex of dimension $m+n-1$, and by Proposition~\ref{prop:bub_extension_shelling} it is shellable.  Then, Theorem~\ref{thm:pure_shellable_homotopy} implies that $\Delta(m,n)$ is homotopy equivalent to a wedge of $(m+n-1)$-dimensional spheres, where the number of spheres involved is $\tilde{\chi}\bigl(\Delta(m,n)\bigr)$ up to sign.  Using \eqref{eq:euler_characteristic} and Theorem~\ref{thm:Ftriangle}, we get
	\begin{align*}
		\tilde{\chi}\bigl(\Delta(m,n)\bigr) & = -f_{\Delta(m,n)}(-1)\\
		& = -F_{m,n}(-1,-1)\\
		& = -\sum_{a=0}^{\min\{m,n\}}\binom{m}{a}\binom{n}{a}(-1)^{a}0^{a}(-1)^{m+n-2a}\\
		& = (-1)^{m+n-1}.\qedhere
	\end{align*}
\end{proof}

Now we have finally gathered all ingredients to prove Theorem~\ref{thm:delta_topology}, which we repeat once again for convenience.

\deltatopology*
\begin{proof}
	The fact that $\Delta(m,n)$ is pure and thin follows from Lemmas~\ref{lem:phi_noncrossing}, \ref{lem:delta_purity} and \ref{lem:delta_thin}.  Proposition~\ref{prop:bub_extension_shelling} implies that $\Delta(m,n)$ is shellable and Proposition~\ref{prop:delta_sphere} states that $\Delta(m,n)$ has the homotopy type of a sphere.
\end{proof}

\section{The $M$-triangle of $\ShufPoset(m,n)$}
    \label{sec:mobius}
Lastly, we consider a bivariate generalization of the characteristic polynomial of the \emph{shuffle} lattice $\ShufPoset(m,n)$, and show that it is related to the previously computed $F$- and $H$-triangles associated with the \emph{bubble} lattice $\Bub(m,n)$.

For a finite poset $\Poset=(P,\leq)$, the \defn{M{\"o}bius function} is the map $\mu_{\Poset}\colon P\times P\to\mathbb{Z}$ defined recursively by
\begin{displaymath}
    \mu_{\Poset}(u,v) \defs \begin{cases}1, & \text{if}\;u=v,\\-\sum_{u\leq r<v}\mu_{\Poset}(u,r), & \text{if}\;u<v,\\0, & \text{otherwise}.\end{cases}
\end{displaymath}
When $\Poset$ is graded and has a unique minimal element and a unique maximal element (denoted by $\hat{0}$ and $\hat{1}$, respectively), then we may define the \defn{(reverse) characteristic polynomial} of a graded poset $\Poset$ by
\begin{displaymath}
	\ch_{\Poset}(q) \defs \sum_{v\in P}\mu_{\Poset}(\hat{0},v)q^{\rk(v)}.
\end{displaymath}
A bivariate generalization of this is the \defn{$M$-triangle}:
\begin{displaymath}
	M_{\Poset}(q,t) \defs \sum_{u\leq v}\mu_{\Poset}(u,v)q^{\rk(u)}t^{\rk(v)}.
\end{displaymath}
These polynomials are related as follows.

\begin{lemma}
	We have
	\begin{align*}
		\ch_{\Poset}(q) & = M_{\Poset}(0,q),\\
		M_{\Poset}(q,t) & = \sum_{u\in P}(qt)^{\rk(u)}\ch_{[u,\hat{1}]}(t).
	\end{align*}
\end{lemma}
\begin{proof}
	This is a straightforward computation.
\end{proof}

\begin{lemma}
	If $\Poset$ is self-dual of rank $k$, then we have
	\begin{displaymath}
		M_{\Poset}(q,t) = (qt)^{k}M_{\Poset}\left(\frac{1}{t},\frac{1}{q}\right).
	\end{displaymath}
\end{lemma}
\begin{proof}
	Suppose that $\xi$ is the dualization map from $\Poset$ to $\Poset^{d}$.  Let $u,v\in P$.  If $u\leq v$, then $\xi(v)\leq\xi(u)$ which implies $\mu_{\Poset}(u,v)=\mu_{\Poset^{d}}\bigl(\xi(v),\xi(u)\bigr)$.   Moreover, the rank of $\xi(u)$ in $\Poset^{d}$ is $k-\rk(u)$.  The $M$-triangle of $\Poset^{d}$ is
	\begin{align*}
		\tilde{M}(q,t) & = \sum_{\xi(v)\leq\xi(u)}\mu_{\Poset^{d}}\bigl(\xi(v),\xi(u)\bigr)q^{\rk(\xi(v))}t^{\rk(\xi(u))}\\
		& = \sum_{u\leq v}\mu_{\Poset}(u,v)q^{k-\rk(v)}t^{k-\rk(u)}\\
		& = (qt)^{k}M_{\Poset}\left(\frac{1}{t},\frac{1}{q}\right).
	\end{align*}
	Since $\Poset$ is self-dual we get $\tilde{M}(q,t)=M_{\Poset}(q,t)$.  This finishes the proof.
\end{proof}

In this section, we are interested in the case where $\Poset=\ShufPoset(m,n)$.  We write $\ch_{m,n}(q)$ instead of $\ch_{\ShufPoset(m,n)}(q)$, $\mu_{m,n}(\ubf,\vbf)$ instead of $\mu_{\ShufPoset(m,n)}(\ubf,\vbf)$ and $M_{m,n}(q,t)$ instead of $M_{\ShufPoset(m,n)}(q,t)$.

\begin{proposition}\label{prop:shuffle_char}
	For $m,n\geq 0$, we have 
	\begin{displaymath}
		\ch_{m,n}(q) = \sum_{a\geq 0}\binom{m}{a}\binom{n}{a}(-q)^{a}(1-q)^{m+n-a}.
	\end{displaymath}
\end{proposition}
\begin{proof}
	We have
	\begin{align*}
		q^{m+n}\ch_{m,n}\left(\frac{1}{q}\right) & = \sum_{\vbf\in\Shuf(m,n)}\mu_{m,n}(\xbf,\vbf)q^{m+n-\rk(\vbf)}\\
		& = \sum_{\vbf\in\Shuf(m,n)}\mu_{m,n}(\xbf,\vbf)q^{\cork(\vbf)}\\
		& = \sum_{a\geq 0}\binom{m}{a}\binom{n}{a}(-1)^{a}(q-1)^{m+n-a},
	\end{align*}
	where the last equality follows from \cite{greene:shuffle}*{Theorem~3.4}.  The claim now follows.
\end{proof}

We wish to emphasize that the reverse characteristic polynomial of the shuffle lattice arises as a specialization of the $H$-triangle associated with the \emph{bubble} lattice.  

\begin{corollary}\label{cor:char_from_h}
	For $m,n\geq 0$, we have
	\begin{displaymath}
		\ch_{m,n}(q) = q^{m+n}H_{m,n}\left(\frac{q-1}{q},\frac{1-2q}{q-1}\right).
	\end{displaymath}
\end{corollary}

The relation from Corollary~\ref{cor:char_from_h} is proven by comparing the explicit formulas from Theorem~\ref{thm:bubble_h_triangle} and Proposition~\ref{prop:shuffle_char}.  We are not aware of a conceptual explanation of this equation.  Such an explanation would be exremely helpful to prove the following relation.

\begin{conjecture}\label{conj:h_to_m}
	For $m,n\geq 0$, we have
	\begin{equation}\label{eq:h_to_m}
		M_{m,n}(q,t) = (1-t)^{m+n}H\left(\frac{t(q-1)}{1-t},\frac{q}{q-1}\right).
	\end{equation}
\end{conjecture}

In view of Conjecture~\ref{conj:h_to_m} we can conjecture an explicit formula for the $M$-triangle of the shuffle lattice.

\begin{conjecture}\label{conj:h_to_m_formula}
	For $m,n\geq 0$, we have
	\begin{align*}
		M_{m,n}(q,t) 
		& = \sum_{a\geq 0} \binom{m}{a}\binom{n}{a}t^{a}(1-t)^{a}(q-1)^{a}(qt-t+1)^{m+n-2a}.
	\end{align*}
\end{conjecture}

As a sanity check, we compute
\begin{displaymath}
	M_{m,n}(0,q) = \sum_{a\geq 0} \binom{m}{a}\binom{n}{a}(-q)^{a}(1-q)^{m+n-a} = \ch_{m,n}(q).
\end{displaymath}

\begin{example}
    By inspection of Figure~\ref{fig:shuffle_21}, we find 
    \begin{align*}
        M_{2,1}(q,t) & =  q^{3}t^{3} - 5q^{2}t^{3} + 5q^{2}t^{2} + 7qt^{3} - 12qt^{2} - 3t^{3} + 5qt + 7t^{2} - 5t + 1\\
        & = (qt-t+1)^{3} + 2(q-1)(1-t)t(qt-t+1)\\
        & = \sum_{a=0}^{1} \binom{2}{a}\binom{1}{a}(q-1)^{a}(1-t)^{a}t^{a}(qt-t+1)^{3-2a},
	\end{align*}
	which confirms Conjecture~\ref{conj:h_to_m_formula}.
\end{example}

Let us denote by $\Delta^{+}(m,n)$ the subcomplex of $\Delta(m,n)$ that consists of all the faces without loops.

\begin{lemma}\label{lem:positive_facets_mobius}
	The number of facets of $\Delta^{+}(m,n)$ is
	\begin{displaymath}
		\binom{m+n}{n} = (-1)^{m+n}\mu_{m,n}(\xbf,\ybf).
	\end{displaymath}
\end{lemma}
\begin{proof}
	In view of \eqref{eq:phi_bijection}, a facet $F_{\wbf}\in\Delta$ contains no loops if and only if $\wbf$ has full support.  The number of shuffle words of full support is clearly $\binom{m+n}{n}$; see also \cite{mcconville.muehle:bubbleI}*{Proposition~3.8}.  Theorem~3.4 in \cite{greene:shuffle} states that $\mu_{m,n}(\xbf,\ybf)=(-1)^{m+n}\binom{m+n}{n}$.
\end{proof}

\bibliography{bib_shuffle}

\end{document}